\documentclass[10pt]{article}

\usepackage{latexsym, amssymb, amsmath, amsthm, amscd}
\usepackage[all,cmtip]{xy}
\usepackage{amsmath}
\usepackage{amsthm}
\usepackage{amssymb}
\usepackage{amsfonts}
\usepackage{amsrefs}
\usepackage{color,authblk}
\usepackage{tikz}
\usepackage{amscd} 
\usepackage{comment}
\usepackage{mathrsfs}
\usepackage{comment}
\usepackage[all]{xy}
\usepackage{graphicx}
\usepackage{authblk}
\usepackage{hyperref}
\usepackage{bbm}
\usepackage{fullpage}
\usepackage[all,cmtip]{xy}
\usepackage{relsize}
\usepackage{amsmath,amscd}
\usepackage{tikz-cd}
\usepackage{tikz}
\usetikzlibrary{matrix}
\usepackage[mathcal]{euscript}
\usepackage{txfonts}

\usepackage{hyperref}

\numberwithin{equation}{section} \DeclareMathSizes{2}{10}{12}{13}
\makeatletter
\newcommand*{\doublerightarrow}[2]{\mathrel{
  \settowidth{\@tempdima}{$\scriptstyle#1$}
  \settowidth{\@tempdimb}{$\scriptstyle#2$}
  \ifdim\@tempdimb>\@tempdima \@tempdima=\@tempdimb\fi
  \mathop{\vcenter{
    \offinterlineskip\ialign{\hbox to\dimexpr\@tempdima+1em{##}\cr
    \rightarrowfill\cr\noalign{\kern.5ex}
    \rightarrowfill\cr}}}\limits^{\!#1}_{\!#2}}}
\newcommand*{\triplerightarrow}[1]{\mathrel{
  \settowidth{\@tempdima}{$\scriptstyle#1$}
  \mathop{\vcenter{
    \offinterlineskip\ialign{\hbox to\dimexpr\@tempdima+1em{##}\cr
    \rightarrowfill\cr\noalign{\kern.5ex}
    \rightarrowfill\cr\noalign{\kern.5ex}
    \rightarrowfill\cr}}}\limits^{\!#1}}}
\makeatother

\newtheorem{thm}{Proposition}[section]
\newtheorem{Thm}[thm]{Theorem}

\newtheorem{lem}[thm]{Lemma}
\newtheorem{defn}[thm]{Definition}

\title{Measurings of Hopf algebroids and morphisms in cyclic (co)homology theories}
\author{Abhishek Banerjee \footnote{Department of Mathematics, Indian Institute of Science, Bangalore. Email: abhishekbanerjee1313@gmail.com} $\qquad\qquad\qquad$ Surjeet Kour \footnote{Department of Mathematics, Indian Institute of Technology, Delhi. Email: koursurjeet@gmail.com}}
\date{}

\begin{document}
\maketitle

\medskip
\begin{abstract} In this paper, we consider measurings between Hopf algebroids and show that they induce morphisms on cyclic homology and cyclic cohomology. We also consider comodule measurings between SAYD modules over Hopf algebroids. These give an enrichment of the global category of SAYD modules over comodules. These measurings also  induce morphisms on cyclic (co)homology of Hopf algebroids with SAYD coefficients, which are compatible with Hopf-Galois maps. Finally, we consider non-$\Sigma$ operads with multiplication. We obtain an enrichment of cyclic unital comp modules over non-$\Sigma$ operads, as well as morphisms on cyclic homology induced by measurings of comp modules over operads with multiplication.
\end{abstract}

\medskip MSC(2020) Subject Classification: 16T15, 16E40, 18D50


\medskip
\section{Introduction} 

Let $k$ be a field and let $Vect_k$ be the category of vector spaces over $k$. Let $R$, $R'$ be $k$-algebras. A measuring from $R$ to $R'$ consists of a $k$-coalgebra $C$ and a $k$-linear map $\psi:C\longrightarrow Vect_k(R,R')$ satisfying the following conditions
\begin{equation}\label{11vs}
\psi(x)(r_1r_2)=\sum \psi(x_{(1)})(r_1)\psi(x_{(2)})(r_2) \qquad \psi(x)(1_R)=\epsilon_C(x)1_{R'}\qquad x\in C, r_1,r_2\in R
\end{equation}  Here, $\Delta_C(x)=\sum x_{(1)}\otimes x_{(2)}$ denotes the coproduct on $C$ and $\epsilon_C:C\longrightarrow k$ denotes the counit. This is a classical notion introduced by
Sweedler \cite{Sweed}, and the pair $(C,\psi)$ captures generalized morphisms of algebras from $R$ to $R'$. For instance, if $x\in C$ is a group like element, i.e., $\Delta_C(x)=x\otimes x$ and $\epsilon_C
(x)=1$, the condition in \eqref{11vs} reduces to that of a unital ring morphism $\psi(x)$ from $R$ to $R'$.  If $R=R'$ and  $x\in C$ is a primitive element, i.e., 
$\Delta_C(x)=x\otimes 1+1\otimes x$, then \eqref{11vs} gives a derivation $\psi(x)$ on $R$. 

\smallskip
Between any $k$-algebras $R$ and $R'$, Sweedler \cite{Sweed}  showed that there is a universal object in the category of coalgebra measurings from $R$ to $R'$. This universal measuring coalgebra, denoted $\mathcal M(R,R')$, plays the role of a generalized hom, giving an enrichment of $k$-algebras over the category of $k$-coalgebras. Accordingly, the universal measuring coalgebra is often called the ``Sweedler hom'' (see, for instance, Anel and Joyal \cite{AJ}). We mention also that the Sweedler hom is closely related to the classical Sweedler dual of an algebra (see Porst and Street \cite{Pors}).  Over the years, measurings have been widely studied in a number of contexts, such as with differential graded algebras (see Anel and Joyal \cite{AJ}), monoids in a braided monoidal category (see Hyland, L\'{o}pez Franco and Vasilakopoulou \cite{V2}, Vasilakopoulou \cite{Vs0}), bialgebras and Hopf algebras (see Grunenfelder
and Mastnak \cite{GM1}, \cite{GM2}), or with entwining structures (see Brzezi\'{n}ski \cite{BrJA}).  For more on this subject, we refer the reader, for instance, to \cite{AJ}, \cite{BanK}, \cite{Bat0}, \cite{Bat}, \cite{BrJA},  \cite{V2}, \cite{V1},\cite{Vs0}, \cite{Vs1}, \cite{Vs2}. For modules over algebras, there is a similar notion of measuring comodules (see \cite{Bat}, \cite{V1}). This consists of a coalgebra measuring $\psi:C\longrightarrow Vect_k(R,R')$, and a right $C$-comodule $D$ equipped with a linear map $\omega:D\longrightarrow Vect_k(M,M')$ satisfying an adaptation of the condition in \eqref{11vs}, where $M$ is a right $R$-module and $M'$ is a right $R'$-module. The latter leads to an enrichment of modules over comodules (see Hyland, L\'{o}pez Franco and Vasilakopoulou \cite{V1}). 

\smallskip
Since coalgebra measurings are like generalized morphisms of algebras, can they be used to induce maps between homology theories? We began studying this question in \cite{BanK} and constructed morphisms in Hochschild homology of algebras. We now take this idea further. 
We ask if the same can be done, for instance, in the Hopf cyclic cohomology of Connes and Moscovici (see \cite{CM1}, \cite{CM2}). Hopf cyclic cohomology and the operators on it are an object of study in their own right. In particular, operators on Hopf cyclic cohomology capture symmetries of ``noncommutative spaces'' arising in a wide variety of situations from foliations to jet bundles and modular Hecke algebras,  which has far reaching implications in noncommutative geometry (see, for instance, \cite{CM1}, \cite{CM3}). 

\smallskip
In this paper, we work with coalgebra measurings in three different situations. Our framework is more exhaustive than in \cite{BanK}, because we are able to include, for the most part, both Hochschild and cyclic (co)homology theories, as well as work with coefficients. The contexts that we study in this paper are as follows.

\smallskip
(1) We begin with Hopf algebroids, which are the generalization of Hopf algebras to noncommutative base rings. We use coalgebra measurings to induce maps between Hochschild and cyclic theories, and comodule measurings to induce maps between these theories with SAYD coefficients. One particular feature is that the morphisms between Hochschild and cyclic (co)homologies
are compatible with Hopf-Galois maps. Another feature is that for commutative Hopf algebroids, we can induce coalgebra measurings of their Hochschild homology algebras equipped with the shuffle product.

\smallskip
(2) We study coalgebra measurings which induce maps on the homology of Lie-Rinehart algebras. We relate these to maps between the periodic cyclic cohomology of their universal enveloping algebras, which carry the structure of left Hopf algebroids. 

\smallskip
(3) Finally, we consider cyclic comp modules in the sense of Kowalzig \cite{Ko4} over non-symmetric operads with multiplication. For comodule measurings between  cyclic comp modules, we obtain induced maps between cyclic homologies. In particular, these measurings induce  maps on cyclic homology of braided commutative Yetter-Drinfeld algebras over a given
Hopf algebroid.

\smallskip We now describe the paper in more detail. We recall that a Hopf algebroid is a datum $\mathcal U=(U,A_L,s_L,t_L,\Delta_L,\epsilon_L,S)$ where $U$ and $A$ are $k$-algebras (we often write
$A_L$ for the ring $A$ and $A_R$ for the opposite ring $A^{op}$), $s_L:A\longrightarrow U$ (resp. $t_L:A^{op}\longrightarrow U$) is known as the source map (resp. the target map), 
$\Delta_L: U\longrightarrow U\otimes_{A_L}U$ is the coproduct,  $\epsilon_L:U\longrightarrow A_L$ is the counit and $S:U\longrightarrow U$ is the antipode (see Definition \ref{D2.2}). By a measuring of Hopf algebroids
from $\mathcal U=(U,A_L,s_L,t_L,\Delta_L,\epsilon_L,S)$ to $\mathcal U'=(U',A_L',s'_L,t'_L,\Delta'_L,\epsilon'_L,S')$, we mean a cocommutative coalgebra $C$ along with a pair $(\Psi,\psi)$ of measurings of algebras
\begin{equation}\label{12intro}
\Psi:C\longrightarrow Vect_k(U,U')\qquad \psi:C\longrightarrow Vect_k(A,A')
\end{equation} satisfying certain compatibility conditions (see Definition \ref{D2.3}).  We construct a universal (cocommutative) measuring coalgebra $\mathcal M_c(\mathcal U,\mathcal U')$ for Hopf algebroids $\mathcal U$, $\mathcal U'$ over $k$. We   use these as generalized hom objects to obtain an enriched category $HALG_k$ of Hopf algebroids (see Theorem \ref{T2.7}) over the symmetric monoidal category $CoCoalg_k$ of cocommutative coalgebras. We show in Section 3 that these measurings induce morphisms on the cyclic (co)homology groups of Hopf algebroids introduced by
Kowalzig and Posthuma \cite{KoP}, and that these are well behaved with respect to Hopf-Galois maps.  If we restrict to Hopf algebroids that are also commutative, we know (see \cite{Ko3}) that their Hochschild homology groups are equipped with a $k$-algebra structure given by the shuffle product. Given a coalgebra measuring between commutative Hopf algebroids as in \eqref{12intro}, we show that the morphisms on the underlying Hochschild homologies
\begin{equation}\label{13intro}
\underline{\Psi}^{hoc}:C\longrightarrow Vect_k(HH_\bullet(\mathcal U),HH_\bullet(\mathcal U'))
\end{equation} give a measuring with respect to the shuffle product structure (see Proposition \ref{P4.1n}). By considering the universal measuring coalgebras of Sweedler between Hochschild homology rings equipped with the shuffle product, we obtain a second enrichment $\widetilde{cHALG_k}$  of commutative Hopf algebroids over $CoCoalg_k$. Further, we construct a $CoCoalg_k$-enriched functor $\tau: cHALG_k\longrightarrow \widetilde{cHALG}_k$, where $cHALG_k$ is the full subcategory of the $CoCoalg_k$-enriched category $HALG_k$ consisting of commutative Hopf algebroids (see Theorem \ref{T4.3}).

\smallskip
Let $\mathcal U$, $\mathcal U'$ be Hopf algebroids and let $P$ (resp. $P'$) be a  stable anti-Yetter Drinfeld module (or SAYD module) over $\mathcal U$ (resp. $\mathcal U'$). By a comodule measuring from $P$ to $P'$, we mean a datum
\begin{equation}\label{14intro} \Psi:C\longrightarrow Vect_k(U,U')\qquad \psi:C\longrightarrow Vect_k(A,A')\qquad \Omega: D\longrightarrow Vect_k(P,P')
\end{equation} satisfying certain conditions, where $(\Psi,\psi)$ is a coalgebra measuring as in \eqref{12intro} and $D$ is a $C$-comodule (see Definition \ref{D5.6}). We construct a universal measuring
comodule $\mathcal Q_C(P,P')$ for the coalgebra measuring $(\Psi,\psi)$. Thereafter, we consider pairs $(\mathcal U,P)$, where $\mathcal U$ is a Hopf algebroid and $P$ is an SAYD module 
over $\mathcal U$. By considering the universal measuring coalgebra $\mathcal M_c(\mathcal U,\mathcal U')$ and taking the corresponding universal measuring comodule $\mathcal Q_{\mathcal M_c(\mathcal U,\mathcal U')}(P,P')$, we show (see Theorem \ref{T5.9hh}) that these pairs form a ``global category of SAYD modules'' that is enriched over the global category of comodules over $k$. The latter consists of pairs
$(C,D)$, where $C$ is a cocommutative $k$-coalgebra and $D$ is a $C$-comodule. The (co)cyclic module corresponding to an SAYD module $P$ over a Hopf algebroid $\mathcal U$ was introduced by Kowalzig and Kr\"{a}hmer in \cite{KoKr}. We show that a comodule measuring as in \eqref{14intro} induces morphisms on the corresponding Hochschild and cyclic (co)homology groups (see Proposition \ref{P6.3h}). Moreover, these induced morphisms are well behaved with respect to Hopf-Galois maps (see Theorem \ref{T6.3}). 

\smallskip
In Section 7, we consider coalgebra measurings of Lie-Rinehart algebras. In particular, we show that a coalgebra measuring between Lie-Rinehart algebras induces morphisms on their homology theories. We also show that such a measuring leads to a coalgebra measuring of their universal enveloping algebras. Additionally, the  morphisms on homology theories of Lie-Rinehart algebras are compatible with the morphisms induced by this coalgebra measuring on the cyclic theory of their universal enveloping algebras. 

\smallskip
In the final part of this paper, we come to non-$\Sigma$ operads with multiplication. We recall that a non-$\Sigma$ operad $(\mathscr O,m,e)$ consists of a collection of vector
spaces $\mathscr O=\{\mathscr O(n)\}_{n\geq 0}$ equipped with composition operations, a multiplication $m\in 
\mathscr O(2)$ and a unit $e\in \mathscr O(0)$. If $(\mathscr O,m,e)$ and $(\mathscr O',m',e')$ are such operads, we consider measurings where $C$ is a cocommutative coalgebra and $\Psi
$ is a family of linear maps
\begin{equation}\label{15intr}
\Psi=\{\Psi_n:C\longrightarrow Vect_k(\mathscr O(n),\mathscr O'(n))\}_{n\geq 0}
\end{equation} satisfying certain conditions with respect to the composition operations, multiplication and unit (see Definition \ref{D7.11h}). We construct a universal (cocommutative) measuring coalgebra $\mathcal M_c(\mathscr O,\mathscr O')$ and use it to obtain an enrichment of  non-$\Sigma$ operads with multiplication over the  symmetric monoidal category $CoCoalg_k$ of cocommutative coalgebras. We then consider pairs $(\mathscr O,\mathscr L)$, where $\mathscr O$ is a non-$\Sigma$ operad with multiplication and $\mathscr L$ is a cyclic unital comp module over $\mathscr O$ in the sense of Kowalzig \cite{Ko4}. If $\mathscr L$ (resp. $\mathscr L')$ is a cyclic unital comp module over $\mathscr O$  (resp. $\mathscr O'$), we   consider comodule measurings 
\begin{equation}\label{16intr}
\Psi=\{\Psi_n:C\longrightarrow Vect_k(\mathscr O(n),\mathscr O'(n))\}_{n\geq 0}\qquad \Omega=\{\Omega_n:D\longrightarrow Vect_k(\mathscr L(n),\mathscr L'(n))\}_{n\geq 0}
\end{equation} where $\Psi=\{\Psi_n\}_{n\geq 0}$ is a measuring from $\mathscr O$ to $\mathscr O'$ in the sense of \eqref{15intr} and $D$ is a $C$-comodule (see Definition \ref{D7.3}). By constructing universal measuring comodules, we show that the pairs $(\mathscr O,\mathscr L)$ form a category enriched over the global category of comodules (see Theorem \ref{T7.7kq}). In 
\cite{Ko4}, Kowalzig introduced a cyclic module corresponding to a cyclic unital comp module over a non-$\Sigma$ operad with multiplication. Accordingly, we show that a comodule measuring as in \eqref{16intr} induces morphisms on cyclic homology groups. 

\smallskip
We conclude by considering measurings $\psi:C\longrightarrow Vect_k(Z,Z')$ between braided commutative Yetter-Drinfeld algebras $Z$, $Z'$ over a given Hopf algebroid $\mathcal U=(U,A_L,s_L,t_L,\Delta_L,\epsilon_L,S)$. We know from 
\cite{Ko4} that there is a non-$\Sigma$ operad $C^\bullet(\mathcal U,Z)$   with multiplication associated to a braided commutative Yetter-Drinfeld algebra $Z$ over $\mathcal U$. We show that a measuring $\psi:C\longrightarrow Vect_k(Z,Z')$ induces a measuring of non-$\Sigma$ operads from $C^\bullet(\mathcal U,Z)$ to $C^\bullet(\mathcal U,Z')$ in the sense of \eqref{15intr}. Further, if $L$ is an anti-Yetter Drinfeld module over $\mathcal U$ such that $L\otimes_{A^{op}}Z$ is stable, then we know from \cite{Ko4} that this determines a cyclic unital comp module $C_\bullet(\mathcal U,L\otimes_{A^{op}}Z)$ over $C^\bullet(\mathcal U,Z)$. Accordingly, we show that a measuring  $\psi:C\longrightarrow Vect_k(Z,Z')$ along with a morphism
$L\longrightarrow L'$ of anti-Yetter Drinfeld modules induces a comodule measuring of cyclic unital comp modules from  $C_\bullet(\mathcal U,L\otimes_{A^{op}}Z)$  to  $C_\bullet(\mathcal U,L'\otimes_{A^{op}}Z')$   in the sense of \eqref{16intr}. In particular, this induces   morphisms on the corresponding cyclic homology groups (see Proposition \ref{P7.9uc}).

\section{Measurings of Hopf algebroids}

Throughout, $k$ is a field and let $Vect_k$ be the category of $k$-vector spaces. Let $A$ be a unital $k$-algebra. In order to define left and right bialgebroids, as well as Hopf algebroids, we will frequently need both the algebra $A$ and its opposite algebra $A^{op}$. For this, we will often write the algebra $A$ as $A_L$, while $A^{op}$ will often be written as $A_R$. 

\smallskip
An $(s,t)$-ring over $A$ consists of a unital $k$-algebra $U$ along with two $k$-algebra morphisms $s:A\longrightarrow U$ and $t:A^{op}\longrightarrow U$ whose images commute in $U$, i.e., $s(a_1)t(a_2)=t(a_2)s(a_1)$ for any $a_1$, $a_2\in U$. The morphisms $s$ and $t$ are often referred to as source and target maps respectively. These morphisms introduce an $(A,A)$-bimodule structure on $U$ given by left multiplication
\begin{equation}\label{2.1eq}
a_1\cdot u\cdot a_2:=s(a_1)t(a_2)u \qquad a_1,a_2 \in A, \textrm{ }u\in U
\end{equation} The left and right $A$-module structures on $U$ in \eqref{2.1eq} allow us to consider the tensor product $U\otimes_AU$. The following subspace of $U\otimes_AU$ is known as the Takeuchi product
\begin{equation}\label{tak}
U\times_AU:=\{\mbox{$\sum u_i\otimes_Au_i'\in U\otimes_AU$ $\vert$ $\sum u_it(a)\otimes_Au_i'=\sum u_i\otimes_Au'_is(a) $, $\forall$ $a\in A$}\}
\end{equation}  It is well known (see, for instance, \cite[$\S$ 2]{KoP}) that the Takeuchi product $U\times_AU$ is a unital subalgebra of $U\otimes_AU$.

\smallskip
From now onwards, we also fix a unital $k$-algebra $U$. The multiplication on $U$ will be denoted by $\mu_U$. 
Since the category of $(A,A)$-bimodules is monoidal, we can consider coalgebra objects in this category. We now recall the notion of a left Hopf algebroid (see, for instance, \cite{BoK}, \cite{KoP}, \cite{Tak}). For several closely related notions, see \cite{Sch1}, \cite{Sch2}. 

\begin{defn}\label{D2.1}
A left bialgebroid $\mathcal U_L:=(U,A_L,s_L,t_L,\Delta_L,\epsilon_L)$  over $k$ consists of the following data:

\smallskip
(1) A unital $k$-algebra $A_L$ 

\smallskip
(2) A unital $k$-algebra $U$ which carries the structure of an $(s_L,t_L)$ ring over $A_L$.

\smallskip
(3) A coalgebra object $(U,\Delta_L:U\longrightarrow U\otimes_{A_L}U,\epsilon_L:U\longrightarrow A_L)$ in the category of $(A_L,A_L)$-bimodules satisfying the following conditions:

\smallskip
$\qquad$  (i) $\Delta_L:U\longrightarrow U\otimes_{A_L}U$ factors through $U\times_{A_L}U\subseteq U\otimes_{A_L}U$ and $U\longrightarrow U\times_{A_L}U$ is a morphism of
unital $k$-algebras.

\smallskip
$\qquad$  (ii) $\epsilon_L(us_L(\epsilon_L(u')))=\epsilon_L(uu')=\epsilon_L(ut_L(\epsilon_L(u')))$ for all $u$, $u'\in U$.

\smallskip
A morphism $(F,f):(U,A_L,s_L,t_L,\Delta_L,\epsilon_L)= \mathcal U_L\longrightarrow \mathcal  U_L'=(U',A_L',s_L',t_L',\Delta_L',\epsilon_L')$ of left bialgebroids consists of a pair of unital $k$-algebra morphisms  $F: U\longrightarrow
U'$ and $f: A_L\longrightarrow A_L'$ such that 
\begin{equation}
F\circ s_L=s'_L\circ f\qquad F\circ t_L=t'_L\circ f\qquad \Delta'_L\circ F=(F\otimes_fF)\circ \Delta_L \qquad f\circ \epsilon_L=\epsilon'_L\circ F
\end{equation} We will denote the category of left bialgebroids over $k$ by $LBialg_k$. 
\end{defn}

If $\mathcal U_L=(U,A_L,s_L,t_L,\Delta_L,\epsilon_L)$  is a left bialgebroid, we employ standard Sweedler notation to write $\Delta_L(u)=\sum u_{(1)}\otimes u_{(2)}$ for any $u\in U$. We also suppress the summation symbol throughout.
We now recall the notion of Hopf algebroid from \cite[Definition 4.1]{BoK}. 

\begin{defn}\label{D2.2}
A Hopf algebroid $\mathcal U=(\mathcal U_L,S)$ over $k$ consists of the following data:

\smallskip
(1) A left bialgebroid $\mathcal U_L=(U,A_L,s_L,t_L,\Delta_L,\epsilon_L)$ over $k$.

\smallskip
(2) An involutive anti-automorphism $S:U\longrightarrow U$ of the $k$-algebra $U$ which satisfies $S\circ t_L=s_L$ as well as
\begin{equation}
S(u_{(1)})_{(1)}u_{(2)}\otimes S(u_{(1)})_{(2)}=1_U\otimes S(u)\qquad S(u_{(2)})_{(1)}\otimes S(u_{(2)})_{(2)}u_{(1)}=S(u)\otimes 1_U
\end{equation} as elements of $U\otimes_{A_L}U$, for all $u\in U$. 

\smallskip
A morphism $(F,f): \mathcal U=(\mathcal U_L,S)\longrightarrow (\mathcal U_L',S')=\mathcal U'$ of Hopf algebroids is a morphism in $LBialg_k$ that also satisfies $S'\circ F=F\circ S$. We will denote the category of Hopf algebroids over $k$ by $HAlg_k$. 
\end{defn}

We remark here that in this paper we will always assume that the antipode on a Hopf algebroid $\mathcal U=(U_L,S)$ is involutive, i.e., $S^2=id$. However, this condition is not part of the  definition 
due to  B\"{o}hm and Szlach\'{a}nyi in \cite{BoK}. Further, it is shown in \cite[Proposition 4.2]{BoK} that a Hopf algebroid $(\mathcal U_L,S)$ is equivalent to a datum consisting of a left bialgebroid and a right bialgebroid connected by an antipode. 

\smallskip
We now recall the classical notion of a coalgebra measuring due to Sweedler \cite{Sweed}. Let $R$, $R'$ be $k$-algebras and $C$ be a $k$-coalgebra.  Then, a $C$-measuring from 
$R$ to $R'$ consists of a morphism $\psi:C\longrightarrow Vect_k(R,R')$ such that
\begin{equation}\label{2.5ft}
\psi(x)(ab)=\sum \psi(x_{(1)})(a)\psi(x_{(2)})(b) \qquad \psi(x)(1_R)=\epsilon_C(x)1_{R'} \qquad \forall\textrm{ }a,b\in R 
\end{equation} where the coproduct $\Delta_C:C\longrightarrow C\otimes C$ is given by $\Delta_C(x)=\sum x_{(1)}\otimes x_{(2)}$ for any $x\in  C$ and $\epsilon_
C:C\longrightarrow k$ is the counit. From now on, we will almost always suppress the summation symbol in Sweedler notation and write $\Delta_C(x)=x_{(1)}\otimes x_{(2)}$. The measuring as in \eqref{2.5ft} is said to be cocommutative if the coalgebra $C$ is cocommutative. In this paper, we will only consider
cocommutative measurings. If $\psi:C\longrightarrow Vect_k(R,R')$ is a coalgebra measuring, we will often write the morphism $\psi(x)\in Vect_k(R,R')$ simply
as $x:R\longrightarrow R'$ for any $x\in  C$. 

\smallskip
We are now ready to introduce the notion of measuring between Hopf algebroids.

\begin{defn}\label{D2.3}
Let  $\mathcal U=(\mathcal U_L,S)=(U,A_L,s_L,t_L,\Delta_L,\epsilon_L,S)$ and
$\mathcal U'=(\mathcal U'_L,S')=(U',A_L',s'_L,t'_L,\Delta'_L,\epsilon'_L,S')$ be Hopf algebroids over $k$. Let $C$ be a cocommutative $k$-coalgebra. A $C$-measuring $(\Psi,\psi)$ from  $\mathcal U$ to $\mathcal U'$ consists of a pair of
measurings 
\begin{equation} \Psi:C\longrightarrow Vect_k(U,U')\qquad \psi:C\longrightarrow Vect_k(A_L,A'_L)
\end{equation} such that the following diagrams commute for any $x\in  C$
\begin{equation}\label{2.7}
\begin{array}{ccc}
\begin{CD}
A_L @>s_L>> U \\
@VxVV @VVxV \\
A_L' @>s'_L>> U'\\
\end{CD} \qquad \qquad &  \begin{CD}
A_L @>t_L>> U \\
@VxVV @VVxV \\
A_L' @>t'_L>> U'\\
\end{CD} &\qquad \qquad \begin{CD}
U @>S>> U \\
@VxVV @VVxV \\
U' @>S'>> U'\\
\end{CD} \\
\end{array}
\end{equation}
\begin{equation}\label{2.8}
\begin{array}{cc}
\begin{CD}
U @>\epsilon_L>> A_L \\
@VxVV @VVxV \\
U' @>\epsilon'_L>> A_L'\\
\end{CD}\qquad \qquad  &\begin{CD}
U@>\Delta_L>> U\otimes_{A_L}U\\
@VxVV @VVxV \\
U' @>\Delta'_L>> U' \otimes_{A'_L}U'\\
\end{CD} \\
\end{array}
\end{equation} where the arrow $x: U\otimes_{A_L}U\longrightarrow  U' \otimes_{A'_L}U'$  is defined by setting $x(u^1\otimes u^2):=x_{(1)}(u^1)\otimes x_{(2)}(u^2)$ for $u^1\otimes u^2
\in U \otimes_{A_L}U $. 

\end{defn}

Before proceeding further, we need to verify the following fact.

\begin{lem}\label{L2.4}
For any $x\in  C$, the morphism $x: U\otimes_{A_L}U\longrightarrow  U' \otimes_{A'_L}U'$  defined by setting $x(u^1\otimes u^2):=x_{(1)}(u^1)\otimes x_{(2)}(u^2)$ for $u^1\otimes u^2
\in U \otimes_{A_L}U $ is well-defined.
\end{lem}

\begin{proof} We consider $u^1,u^2\in U$ and $a\in A_L$. Using the fact that $ \Psi:C\longrightarrow Vect_k(U,U')$ is a measuring
and applying the conditions in \eqref{2.7} and \eqref{2.8}, we see that for any $x\in C$ we have
\begin{equation}
\begin{array}{ll}
x((u^1\cdot a)\otimes u^2)=x(t_L(a)u^1\otimes u^2)&=x_{(1)}(t_L(a)u^1)\otimes x_{(2)}(u^2)
=x_{(1)}(t_L(a))x_{(2)}(u^1)\otimes x_{(3)}(u^2)\\
&=t_L'(x_{(1)}(a))x_{(2)}(u^1)\otimes x_{(3)}(u^2)\\
&=t_L'(x_{(2)}(a))x_{(1)}(u^1)\otimes x_{(3)}(u^2)\quad \mbox{(because $C$ is cocommutative)}\\
&= x_{(1)}(u^1)\cdot x_{(2)}(a)\otimes x_{(3)}(u^2)=x_{(1)}(u^1)\otimes x_{(2)}(a)\cdot x_{(3)}(u^2)\\ &=x_{(1)}(u^1)\otimes s_L'(x_{(2)}(a))x_{(3)}(u^2)
=x_{(1)}(u^1)\otimes x_{(2)}(s_L(a))x_{(3)}(u^2)\\
&=x_{(1)}(u^1)\otimes x_{(2)}(s_L(a)u^2)=x_{(1)}(u^1)\otimes x_{(2)}(a\cdot u^2)\\
&= x(u^1\otimes (a\cdot u^2))
\end{array}
\end{equation}

\end{proof}

If $\mathcal U=(\mathcal U_L,S)$ and $\mathcal U'=(\mathcal U_L',S')$ are Hopf algebroids over $k$, we now consider the subspace
\begin{equation}\label{210y}
V(\mathcal U,\mathcal U') \subseteq Vect_k(U,U')\times Vect_k(A_L,A_L')
\end{equation} given by setting
\begin{equation}\label{211y}
V(\mathcal U,\mathcal U'):=\{\mbox{$(F,f)$ $\vert$  $Fs_L=s_L'f$, $Ft_L=t_L'f$, $FS=S'F$ and $f\epsilon_L=\epsilon'_LF$ }\}
\end{equation} We note that a measuring from $\mathcal U$ to $\mathcal U'$  by means of a cocommutative coalgebra $C$ has an underlying morphism $(\Psi,\psi):C\longrightarrow V(\mathcal U,\mathcal U')$. 

\smallskip
Let $Coalg_k$ denote the category of $k$-coalgebras. We know that the forgetful
functor $Coalg_k\longrightarrow Vect_k$ has a right adjoint $\mathfrak C:Vect_k\longrightarrow Coalg_k$. In other
words, we have natural isomorphisms
\begin{equation}\label{adj2}
Vect_k(C,V)\cong Coalg_k(C,\mathfrak C(V))
\end{equation} for any $k$-coalgebra $C$ and any $k$-vector space $V$.

\begin{thm}\label{P2.5}
Let $\mathcal U=(\mathcal U_L,S)$ and $\mathcal U'=(\mathcal U_L',S')$ be Hopf algebroids over $k$. Then, there exists a cocommutative coalgebra $\mathcal M_c(\mathcal U,\mathcal U')$ and a measuring
$(\Phi,\phi):\mathcal M_c(\mathcal U,\mathcal U')\longrightarrow V( \mathcal U, \mathcal U')$ satisfying the following universal property: given any measuring $(\Psi,\psi):C\longrightarrow V(\mathcal U,\mathcal U')$ with a cocommutative coalgebra $C$, there exists a unique morphism $\xi:C\longrightarrow \mathcal M_c(\mathcal U,\mathcal U')$ of coalgebras making the following diagram commutative
\begin{equation}
\xymatrix{
\mathcal M_c(\mathcal U,\mathcal U')\ar[rr]^{(\Phi,\phi)} && V(\mathcal U,\mathcal U') \\
&C\ar[ul]^\xi\ar[ur]_{(\Psi,\psi)}&\\
}
\end{equation}
\end{thm}

\begin{proof}
We set $V:=V(\mathcal U,\mathcal U') $ and consider the canonical morphism $\pi(V):\mathfrak C(V)\longrightarrow V\subseteq Vect_k(U,U')\times Vect_k(A_L,A_L')$ from the cofree coalgebra $\mathfrak C(V)$ induced by the adjunction
in \eqref{adj2}. We now set $\mathcal M_c(\mathcal U,\mathcal U'):=\sum D$, where the sum is taken over all cocommutative subcoalgebras of $\mathfrak C(V)$ such that the restriction
$\pi(V)\vert_D:D\longrightarrow V=V(\mathcal U,\mathcal U') $ is a measuring. It is clear that this sum is still a cocommutative coalgebra, and that the restriction $(\Phi,
\phi):=\pi(V)|_{\mathcal M_c(\mathcal U,\mathcal U')}$ gives a measuring from $\mathcal U$ to $\mathcal U'$.

\smallskip
In general, if $(\Psi,\psi):C\longrightarrow V=V(\mathcal U,\mathcal U') $ is a cocommutative measuring, the adjunction in \eqref{adj2} shows that it factors through $\xi:C\longrightarrow \mathfrak C(V)$. Then, $\xi(C)\subseteq \mathfrak C(V)$ is a cocommutative coalgebra such that the restriction $\pi(V)|_{\xi(C)}$ is a measuring. By definition, it follows that $\xi(C)\subseteq \mathcal M_c(\mathcal U,\mathcal U')$. This proves the result. 
\end{proof}

From \eqref{210y} and \eqref{211y} it is clear that given Hopf algebroids $\mathcal U=(\mathcal U_L,S)$, $\mathcal U'=(\mathcal U_L',S')$ and $\mathcal U''=(\mathcal U_L'',S'')$, the composition of morphisms induces a canonical map
\begin{equation}\label{comp2}
V(\mathcal U,\mathcal U') \otimes V(\mathcal U',\mathcal U'')\overset{\circ}{\longrightarrow}V(\mathcal U,\mathcal U'')  
\end{equation} We denote by $CoCoalg_k$ the category of cocommutative coalgebras over $k$. We know that this category is symmetric monoidal and our objective is to show that the category $HAlg_k$ of Hopf algebroids is enriched over $CoCoalg_k$. For this we need the following result.

\begin{thm}\label{P2.6}
Let  $\mathcal U=(\mathcal U_L,S)$, $\mathcal U'=(\mathcal U_L',S')$ and $\mathcal U''=(\mathcal U_L'',S'')$ be Hopf algebroids over $k$. Suppose that we have  a measuring $(\Psi,\psi):C\longrightarrow V(\mathcal U,\mathcal U')$   and a measuring 
$(\Psi',\psi'):C'\longrightarrow V(\mathcal U',\mathcal U'')$. Then, the following 
\begin{equation}\label{xcomp2}
(\Psi',\psi')\circ (\Psi,\psi): C\otimes C'\xrightarrow{(\Psi,\psi)\otimes (\Psi',\psi')}V(\mathcal U,\mathcal U') \otimes V(\mathcal U',\mathcal U'')\overset{\circ}{\longrightarrow}V(\mathcal U,\mathcal U'')  
\end{equation} 
determines a measuring  from $\mathcal U$ to $\mathcal U''$.
\end{thm}

\begin{proof}
It is easy to verify that the compositions 
\begin{equation}
\begin{array}{c} C\otimes C'\xrightarrow{\Psi\otimes \Psi'}Vect_k(U,U')\otimes Vect_k(U',U'')\xrightarrow{\quad\circ\quad}Vect_k(U,U'')\\
C\otimes C'\xrightarrow{\psi\otimes \psi'}Vect_k(A_L,A_L')\otimes Vect_k(A_L',A_L'')\xrightarrow{\quad\circ\quad}Vect_k(A_L,A_L'')\\
\end{array}
\end{equation} give coalgebra measurings from $U$ to $U''$ and from $A_L$ to $A_L''$ respectively. For $x\otimes x'\in C\otimes C'$ and $u\in U$, we also see that
\begin{equation}
\begin{array}{ll}
\Delta_L''((x\otimes x')(u))=\Delta_L''(x'(x(u)))&=x'_{(1)}(x(u)_{(1)})\otimes x'_{(2)}(x(u)_{(2)}) \\
&=x'_{(1)}(x_{(1)}(u_{(1)}))\otimes x'_{(2)}(x_{(2)}(u_{(2)}))\\
&=(x'\otimes x)_{(1)}(u_{(1)})\otimes (x'\otimes x)_{(2)}(u_{(2)})\\
\end{array}
\end{equation} It is also clear that the morphism in \eqref{xcomp2} satisfies all the other conditions in Definition \ref{D2.3}. This proves the result. 
\end{proof}

\begin{Thm}\label{T2.7}
The category  $HAlg_k$  of Hopf algebroids is enriched over the category $CoCoalg_k$ of cocommutative $k$-coalgebras.
\end{Thm}

\begin{proof}
Given Hopf algebroids $\mathcal U=(\mathcal U_L,S)$ and $\mathcal U'=(\mathcal U_L',S')$, we consider the ``hom object'' $\mathcal M_c(\mathcal U,\mathcal U')$ which lies in $CoCoalg_k$. The composition of these hom objects is obtained as follows: if $\mathcal U$, $\mathcal U'$ and $\mathcal U''$  are Hopf algebroids, we obtain as in Proposition \ref{P2.6} a measuring
\begin{equation}\label{xzcom}
\mathcal M_c(\mathcal U,\mathcal U')\otimes \mathcal M_c(\mathcal U',\mathcal U'')\longrightarrow V(\mathcal U,\mathcal U'')  
\end{equation} Applying the universal property in Proposition \ref{P2.5}, we now have a morphism of coalgebras $\mathcal M_c(\mathcal U,\mathcal U')\otimes \mathcal M_c(\mathcal U',\mathcal U'')\longrightarrow \mathcal M_c(\mathcal U,\mathcal U'')  $. The unit object in $CoCoalg_k$ is $k$ treated as a coalgebra over itself. Then, for any Hopf algebroid
$\mathcal U=(\mathcal U_L,S)=(U,A_L,s_L,t_L,\Delta_L,\epsilon_L,S)$, we have a unit map \begin{equation} k\longrightarrow  V(\mathcal U,\mathcal U)\subseteq Vect_k(U,U)\times Vect_k(A_L,A_L) \qquad t\mapsto (t\cdot id_{U},t\cdot id_{A_L})\end{equation} Again from the universal property in Proposition \ref{P2.5},  this induces a morphism $k\longrightarrow \mathcal M_c(\mathcal U,\mathcal U)$ of cocommutative coalgebras. Together with the composition
of hom objects in \eqref{xzcom}, we see that $HAlg_k$ is enriched over $CoCoalg_k$.
\end{proof}

From now onwards, we will denote by $HALG_k$ the category of Hopf algebroids enriched over the symmetric monoidal category $CoCoalg_k$ of cocommutative $k$-algebras. This enriched category will be needed in Section 4. 

\section{Morphisms on cyclic (co)homology and Hopf-Galois maps}

Let $\mathcal U=(\mathcal U_L,S)=(U,A_L,s_L,t_L,\Delta_L,\epsilon_L,S)$  be a Hopf algebroid over $k$. We now recall from \cite[$\S$ 2]{KoP} the cocyclic module $C^\bullet(\mathcal U)$ that computes the cyclic cohomology of the Hopf algebroid 
$\mathcal U$. For $n\geq 1$, we put
\begin{equation}\label{cocych}
C^n(\mathcal U):=\underset{\mbox{$n$-times}}{\underbrace{U\otimes_{A_L}\otimes \dots \otimes_{A_L}U}}
\end{equation} using the $(A_L,A_L)$-bimodule structure as in \eqref{2.1eq} and set $C^0(\mathcal U):=A_L$. For $n\geq 1$, the coface maps $\delta_i:C^n(\mathcal U)\longrightarrow C^{n+1}(\mathcal U)$ are defined by
\begin{equation}\label{degco1}
\delta_i(u^1\otimes ... \otimes u^n):=\left\{
\begin{array}{ll}
1\otimes u^1\otimes ...\otimes u^n & \mbox{if $i=0$}\\
u^1\otimes ....\otimes \Delta_Lu^i\otimes...\otimes u^n & \mbox{if $1\leq i\leq n$}\\
u^1\otimes ....\otimes u^n\otimes 1&\mbox{if $i=n+1$}\\
\end{array}\right.
\end{equation} For $n=0$, there are two maps 
$\delta_0:=t_L:C^0(\mathcal U)=A_L\longrightarrow C^1(\mathcal U)=U$ and   
$\delta_1:=s_L:C^0(\mathcal U)=A_L\longrightarrow C^1(\mathcal U)=U
$. The codegeneracy maps $\sigma_i:C^n(\mathcal U)\longrightarrow C^{n-1}(\mathcal U)$ are given by
\begin{equation}\label{cofac}
\sigma_i(u^1\otimes ... \otimes u^n):=u^1\otimes ...\otimes \epsilon_L(u^{i+1})\cdot u^{i+2}\otimes ...\otimes u^n\qquad 0\leq i\leq n-1
\end{equation} The cocyclic operator $\tau_n:C^n(\mathcal U)\longrightarrow C^n(\mathcal U)$ is defined by setting
\begin{equation}\label{cocy2}
\tau_n(u^1\otimes ...\otimes u^n):=(S(u^1)_{(1)}\cdot u^2)\otimes ....\otimes (S(u^1)_{(n-1)}\cdot u^n)\otimes S(u^1)_{(n)}
\end{equation} Since we have assumed that the antipode $S$ is involutive, it follows from \cite[Theorem 2.1]{KoP} that  $C^\bullet(\mathcal U)$ is indeed a cocyclic module. We will denote by $HC^\bullet(\mathcal U)$ the   cyclic cohomology groups of the Hopf algebroid $\mathcal U$. The Hochschild cohomology groups of the Hopf algebroid $\mathcal U$ will then be denoted by $HH^\bullet(\mathcal U)$. 

\smallskip
Let $\mathcal U$, $\mathcal U'$ be Hopf algebroids and let $(\Psi,\psi):C\longrightarrow V(\mathcal U,\mathcal U')$ be a cocommutative measuring from $\mathcal U$ to $\mathcal U'$. For each $x\in  C$, we  now define a family of morphisms
\begin{equation}\label{35e}
\overline{\Psi}^n(x):C^n(\mathcal U)\longrightarrow C^n(\mathcal U')\qquad \overline\Psi^n(x)(u^1\otimes ...\otimes u^n):=x(u^1\otimes ...\otimes u^n)=x_{(1)}(u^1)\otimes ...\otimes x_{(n)}(u^n)\qquad\forall \textrm{ }n\geq 0
\end{equation} We now prove the first main result of this section.

\begin{thm}\label{P3.1b}
Let $C$ be a cocommutative coalgebra and let $(\Psi,\psi):C\longrightarrow V(\mathcal U,\mathcal U')$ be a measuring of Hopf algebroids. For each $x\in  C$, the family
$\{\overline\Psi^n(x):C^n(\mathcal U)\longrightarrow C^n(\mathcal U')\}_{n\geq 0}$ gives a morphism of cocyclic modules. In particular, we have  induced morphisms
\begin{equation}\overline\Psi^\bullet_{hoc}(x):HH^\bullet(\mathcal U)\longrightarrow HH^\bullet(\mathcal U')\qquad \overline\Psi^\bullet_{cy}(x):HC^\bullet(\mathcal U)\longrightarrow HC^\bullet(\mathcal U')
\end{equation} on Hochschild and cyclic cohomologies for each $x\in  C$. 
\end{thm}

\begin{proof} For each $x\in  C$, we start by showing that $\overline\Psi^{n+1}(x)\circ \delta_i=\delta'_i\circ \overline\Psi^{n}(x):C^n(\mathcal U)\longrightarrow 
C^{n+1}(\mathcal U')$, where $\delta_i$ and $\delta'_i$ are the coface maps on the respective cocyclic modules $C^\bullet(\mathcal U)$ and $C^\bullet(\mathcal U')$. If $i=0$ or $i=n+1$, this is immediately clear from the definition in \eqref{degco1} and the action in \eqref{35e}. For $1\leq i\leq n$, we see that
\begin{equation*}
\begin{array}{ll}
\overline\Psi^{n+1}(x)\circ \delta_i(u^1\otimes ...\otimes u^n)&=\overline\Psi^{n+1}(x)(u^1\otimes ....\otimes \Delta_Lu^i\otimes...\otimes u^n)\\
&=x_{(1)}(u^1)\otimes ... \otimes x_{(i)}(u^i_{(1)})\otimes x_{(i+1)}(u^i_{(2)})\otimes ....\otimes x_{(n+1)}(u^n)\\
&=x_{(1)}(u^1)\otimes \Delta_L(x_{(i)}(u^i))\otimes ... \otimes x_{(n)}(u^n)=\delta'_i\circ \overline\Psi^{n}(x)(u^1\otimes ...\otimes u^n)\\
\end{array}
\end{equation*} 
Next, we verify that $\overline\Psi^{n-1}(x)\circ \sigma_i=\sigma'_i\circ \overline\Psi^{n}(x)$, where  $\sigma_i$ and $\sigma'_i$ are the codegeneracies on the respective cocyclic modules $C^\bullet(\mathcal U)$ and $C^\bullet(\mathcal U')$. 
\begin{equation*}
\begin{array}{ll}
\overline\Psi^{n-1}(x)\circ \sigma_i(u^1\otimes ...\otimes u^n)&=\overline\Psi^{n-1}(x)(u^1\otimes ...\otimes \epsilon_L(u^{i+1})\cdot u^{i+2}\otimes ...\otimes u^n)\\
&=x_{(1)}(u^1)\otimes ...\otimes x_{(i+1)}(\epsilon_L(u^{i+1})\cdot u^{i+2}) \otimes ...\otimes x_{(n-1)}(u^n)\\
&=x_{(1)}(u^1)\otimes ...\otimes x_{(i+1)}(\epsilon_L(u^{i+1}))\cdot x_{(i+2)}(u^{i+2}) \otimes ...\otimes x_{(n)}(u^n)\\
&=x_{(1)}(u^1)\otimes ...\otimes \epsilon_L'(x_{(i+1)}(u^{i+1}))\cdot x_{(i+2)}(u^{i+2}) \otimes ...\otimes x_{(n)}(u^n)\\
&=\sigma'_i\circ \overline\Psi^{n}(x)(u^1\otimes ...\otimes u^n)\\
\end{array}
\end{equation*} Finally, we show that $\overline\Psi^n(x)\circ \tau_n=\tau'_n\circ \overline\Psi^n(x)$, where $\tau_n$ and $\tau'_n$ are the cocyclic operators on the respective cocyclic modules $C^\bullet(\mathcal U)$ and $C^\bullet(\mathcal U')$:
\begin{equation*}
\begin{array}{ll}
\overline\Psi^{n}(x)\circ \tau_n(u^1\otimes ...\otimes u^n)&=\overline\Psi^{n}(x)((S(u^1)_{(1)}\cdot u^2)\otimes ....\otimes (S(u^1)_{(n-1)}\cdot u^n)\otimes S(u^1)_{(n)})\\
&= x_{(1)}((S(u^1)_{(1)}\cdot u^2))\otimes ....\otimes x_{(n-1)}((S(u^1)_{(n-1)}\cdot u^n)\otimes x_{(n)}(S(u^1)_{(n)})\\
&=x_{(1)}(S(u^1)_{(1)})\cdot x_{(2)}(u^2)\otimes ....\otimes x_{(2n-3)}(S(u^1)_{(n-1)})\cdot x_{(2n-2)}(u^n)\otimes x_{(2n-1)}(S(u^1)_{(n)})\\
&=x_{(1)}(S(u^1)_{(1)})\cdot x_{(n+1)}(u^2)\otimes ....\otimes x_{(n-1)}(S(u^1)_{(n-1)})\cdot x_{(2n-1)}(u^n)\otimes x_{(n)}(S(u^1)_{(n)})\\
&=(x_{(1)}(S(u^1)))_{(1)}\cdot x_{(2)}(u^2)\otimes ....\otimes (x_{(1)}(S(u^1)))_{(n-1)} \cdot x_{(n)}(u^n)\otimes (x_{(1)}(S(u^1)))_{(n)} \\
&=(S'(x_{(1)}(u^1)))_{(1)}\cdot x_{(2)}(u^2)\otimes ....\otimes (S'(x_{(1)}(u^1)))_{(n-1)} \cdot x_{(n)}(u^n)\otimes (S'(x_{(1)}(u^1)))_{(n)} \\
&=\tau'_n\circ \overline\Psi^n(x)(u^1\otimes ...\otimes u^n)\\
\end{array}
\end{equation*}
\end{proof}

We continue with a Hopf algebroid $\mathcal U=(\mathcal U_L,S)=(U,A_L,s_L,t_L,\Delta_L,\epsilon_L,S)$. As mentioned in Section 2, we set $A_R:=A_L^{op}=A^{op}$. Following \cite[$\S$ 4]{BoK}, we also set
\begin{equation}\label{r3.6}
 s_R:=t_L\qquad t_R:=S\circ t_L=s_L\qquad  \epsilon_R:=\epsilon_L\circ S:U\longrightarrow A_R
\end{equation} Then, $U$ becomes an $(A_R,A_R)$-bimodule by left and right multiplication as follows
\begin{equation}\label{r3.7rl}
a_1\cdot u\cdot a_2:=s_R(a_1)us_R(a_2)
\end{equation}

Now let
$\mathcal U$, $\mathcal U'$ be Hopf algebroids and consider $(F,f)\in V(\mathcal U,\mathcal U')$. From the conditions in \eqref{211y} and the definitions in \eqref{r3.6}, we already have
\begin{equation}\label{311y}
Fs_R=s_R'f\qquad Ft_R=t_R'f\qquad FS=S'F\qquad f\epsilon_R=\epsilon_R'F
\end{equation} We now need the following result.

\begin{lem}\label{L3.2}
Let $C$ be a cocommutative coalgebra and $(\Psi,\psi):C\longrightarrow V(\mathcal U,\mathcal U')$ a measuring of Hopf algebroids. Then, for each $x\in  C$, there is a well defined morphism
\begin{equation}\label{m311}
x:U\otimes_{A_R}U\longrightarrow U'\otimes_{A_R'}U'\qquad u^1\otimes u^2\mapsto x_{(1)}(u^1)\otimes x_{(2)}(u^2)
\end{equation} 
\end{lem}

\begin{proof}
We consider $u^1,u^2\in U$ and $a\in A_R$. Using the fact that $ \Psi:C\longrightarrow Vect_k(U,U')$  is a  measuring
and applying the conditions in \eqref{311y}, we see that for any $x\in C$ we have
\begin{equation*}
\begin{array}{ll}
x((u^1\cdot a)\otimes u^2)=x(u^1s_R(a)\otimes u^2)&=x_{(1)}(u^1s_R(a))\otimes x_{(2)}(u^2)
=x_{(1)}(u^1)x_{(2)}(s_R(a))\otimes x_{(3)}(u^2)\\
&=x_{(1)}(u^1)s_R'(x_{(2)}(a))\otimes x_{(3)}(u^2)\\
&=x_{(1)}(u^1)\cdot x_{(2)}(a)\otimes x_{(3)}(u^2)\\ 
&= x_{(1)}(u^1)\otimes x_{(2)}(a)\cdot x_{(3)}(u^2)= x_{(1)}(u^1)\otimes s_R'( x_{(2)}(a))x_{(3)}(u^2)\\
&= x_{(1)}(u^1)\otimes x_{(2)}(s_R(a))x_{(3)}(u^2)\\
&=x_{(1)}(u^1)\otimes x_{(2)}(s_R(a)u^2)=x_{(1)}(u^1)\otimes x_{(2)}(a\cdot u^2)=x(u^1\otimes (a\cdot u^2))\\
\end{array}
\end{equation*} It follows that the morphism in \eqref{m311} is well defined.  
\end{proof} 

We now recall from \cite[$\S$ 2.3.1]{KoP} the cyclic module $C_\bullet(\mathcal U)$ defining the cyclic homology of a Hopf algebroid $\mathcal U$. For $n\geq 0$, we use the bimodule 
structure in \eqref{r3.7rl} to set 
\begin{equation}\label{ych}
C_n(\mathcal U):=\underset{\mbox{$n$-times}}{\underbrace{U\otimes_{A_R}\otimes \dots \otimes_{A_R}U}}
\end{equation} and $C_0(\mathcal U):=A_R$. The face maps $d_i:C_n(\mathcal U)\longrightarrow C_{n-1}(\mathcal U)$ are defined by setting
\begin{equation}\label{faccy}
d_i(u^1\otimes ... \otimes u^n):=\left\{\begin{array}{ll}
\epsilon_R(u^1) \cdot u^2\otimes ...\otimes u^n &\mbox{if $i=0$}\\
u^1\otimes ...\otimes u^iu^{i+1}\otimes...\otimes u^n & \mbox{if $1\leq i\leq n-1$}\\
u^1\otimes ...\otimes u^{n-1} \cdot \epsilon_R(S(u^n)) &\mbox{if $i=n$}\\
\end{array}\right.
\end{equation}
The degeneracies $s_i:C_n(\mathcal U)\longrightarrow C_{n+1}(\mathcal U)$ are defined as
\begin{equation}\label{degvc}
s_i(u^1\otimes ... \otimes u^n):=\left\{
\begin{array}{ll}
1\otimes u^1\otimes ...\otimes u^n & \mbox{if $i=0$}\\
u^1\otimes ...\otimes u^i\otimes 1\otimes u^{i+1}\otimes....\otimes u^n & \mbox{if $1\leq i\leq n$} \\
\end{array}\right.
\end{equation} The cyclic operators $t_n:C_n(\mathcal U)\longrightarrow C_{n}(\mathcal U)$ are given by 
\begin{equation}\label{cyt4}
t_n(u^1\otimes ...\otimes u^n):=S(u^1_{(2)}...u_{(2)}^{n-1}u^n)\otimes u^1_{(1)}\otimes u^2_{(1)}\otimes ...\otimes u_{(1)}^{n-1}
\end{equation}  The Hochschild homology groups of the Hopf algebroid $\mathcal U$ will then be denoted by $HH_\bullet(\mathcal U)$ and the cyclic homology groups by $HC_\bullet(\mathcal U)$.
We will now prove the homological counterpart for Proposition \ref{P3.1b}.

\begin{thm}\label{P3.3c}
Let $C$ be a cocommutative coalgebra and let $(\Psi,\psi):C\longrightarrow V(\mathcal U,\mathcal U')$ be a measuring of Hopf algebroids. For each $x\in  C$, the family
\begin{equation} \label{315cl} \underline\Psi_n(x):C_n(\mathcal U)\longrightarrow C_n(\mathcal U') \qquad u^1\otimes ...\otimes u^n\mapsto x(u^1\otimes ...\otimes u^n)=x_{(1)}(u^1)\otimes ... \otimes x_{(n)}(u^n)
\end{equation} for $n\geq 0$ gives a morphism of cyclic modules. In particular, we have  induced morphisms
\begin{equation}\underline\Psi_\bullet^{hoc}(x):HH_\bullet(\mathcal U)\longrightarrow HH_\bullet(\mathcal U')\qquad \underline\Psi_\bullet^{cy}(x):HC_\bullet(\mathcal U)\longrightarrow HC_\bullet(\mathcal U')
\end{equation}
 on Hochschild and cyclic homologies for each $x\in  C$. 
\end{thm}

\begin{proof} As a consequence of Lemma \ref{L3.2}, we know that the morphisms in $\underline\Psi_n(x)$  in \eqref{315cl} are well defined.
Using the properties in  \eqref{311y} and the fact that $ \Psi:C\longrightarrow Vect_k(U,U')$ is a measuring, it may easily be verified that the maps $\underline\Psi_\bullet(x)$ commute with the respective face maps and degeneracy maps on the cyclic modules $C_\bullet(\mathcal U)$ and $C_\bullet(\mathcal U')$. Moreover, if $t_n$ and
$t_n'$ are the respective cyclic operators on  $C_\bullet(\mathcal U)$ and $C_\bullet(\mathcal U')$, we have for each $x\in  C$
\begin{equation*}
\begin{array}{ll}
x(t_n(u^1\otimes ...\otimes u^n))&=x(S(u^1_{(2)}...u_{(2)}^{n-1}u^n)\otimes u^1_{(1)}\otimes u^2_{(1)}\otimes ...\otimes u_{(1)}^{n-1})\\
&= x_{(1)}(S(u^1_{(2)}...u_{(2)}^{n-1}u^n) )\otimes x_{(2)}(u^1_{(1)})\otimes x_{(3)}(u^2_{(1)})\otimes ...\otimes x_{(n)}(u_{(1)}^{n-1}) \\
&=  S'(x_{(1)}(u^1_{(2)})...x_{(n-1)}(u_{(2)}^{n-1})x_{(n)}(u^n)) \otimes x_{(n+1)}(u^1_{(1)})\otimes  ...\otimes x_{(2n-1)}(u_{(1)}^{n-1}) \\
&=  S'(x_{(2)}(u^1_{(2)})...x_{(2n-2)}(u_{(2)}^{n-1})x_{(2n-1)}(u^n)) \otimes x_{(1)}(u^1_{(1)})\otimes  ...\otimes x_{(2n-3)}(u_{(1)}^{n-1}) \qquad \mbox{(as $C$ is cocommutative)}\\
&=  S'(x_{(1)}(u^1)_{(2)} ... x_{(n-1)}(u^{n-1})_{(2)}x_{(n)}(u^n)) \otimes x_{(1)}(u^1)_{(1)}\otimes  ...\otimes x_{(n-1)}(u^{n-1})_{(1)}\\
&=t_n'(x_{(1)}(u^1)\otimes ... \otimes x_{(n)}(u^n))\\
\end{array}
\end{equation*}
\end{proof} 

Our final aim in this section is to show that the morphisms induced by a measuring of Hopf algebroids are well behaved with respect to cyclic duality. More precisely, we know from
\cite[$\S$ 2.3.3]{KoP} that there are Hopf-Galois maps
\begin{equation}\label{hg3}
\xi_n(\mathcal U): C_n(\mathcal U)\overset{\cong}{\longrightarrow}C^n(\mathcal U)\qquad u^1\otimes ...\otimes u^n\mapsto u^1_{(1)}\otimes u^1_{(2)}u^2_{(1)}\otimes u^1_{(3)}u^2_{(2)}u^3_{(1)}\otimes ...\otimes u^1_{(n)}u^2_{(n-1)}....u_{(2)}^{n-1}u^n
\end{equation}
inducing isomorphisms between $C_\bullet(\mathcal U)$ and $C^\bullet(\mathcal U)$. We now have the following result.

\begin{thm}\label{P3.4x}
Let $C$ be a cocommutative coalgebra and let $(\Psi,\psi):C\longrightarrow V(\mathcal U,\mathcal U')$ be a measuring of Hopf algebroids. Then for each $x\in  C$, the following diagram commutes
\begin{equation}\label{320cd}
\begin{CD}
 C_n(\mathcal U) @>\xi_n(\mathcal U)>>C^n(\mathcal U)\\
 @V \underline\Psi_n(x)VV @VV \overline\Psi^n(x)V\\
  C_n(\mathcal U') @>\xi_n(\mathcal U')>>C^n(\mathcal U')\\
\end{CD}
\end{equation}
\end{thm}
\begin{proof}
We put $N:=n(n+1)/2$. Using the fact that $\Psi:C\longrightarrow Vect_k(U,U')$ is a measuring and that $C$ is cocommutative we have
\begin{equation*}
\begin{array}{ll}
x(\xi_n(\mathcal U)(u^1\otimes ...\otimes u^n)) &=x( u^1_{(1)}\otimes u^1_{(2)}u^2_{(1)}\otimes u^1_{(3)}u^2_{(2)}u^3_{(1)}\otimes ...\otimes u^1_{(n)}u^2_{(n-1)}....u_{(2)}^{n-1}u^n)\\
&= x_{(1)}(u^1_{(1)})\otimes x_{(2)}(u^1_{(2)})x_{(3)}(u^2_{(1)})\otimes  ...\otimes x_{(N+1-n)}(u^1_{(n)})....x_{(N-1)}(u_{(2)}^{n-1})x_{(N)}(u^n)\\
&= x_{(1)}(u^1_{(1)})\otimes x_{(2)}(u^1_{(2)})x_{(n+1)}(u^2_{(1)})\otimes  ...\otimes x_{(n)}(u^1_{(n)})....x_{(N-1)}(u_{(2)}^{n-1})x_{(N)}(u^n)\\
&=\xi_n(\mathcal U')(x_{(1)}(u^1)\otimes ...\otimes x_{(n)}(u^n))\\
\end{array}
\end{equation*} This proves the result.
\end{proof}

\section{Shuffle products and the enrichment of the category of commutative Hopf algebroids}

We recall from Section 2 the category $HALG_k$ of Hopf algebroids over $k$, enriched over the symmetric monoidal category $CoCoalg_k$ of cocommutative $k$-coalgebras. If $cHALG_k$ denotes the full subcategory of $HALG_k$ consisting of commutative Hopf algebroids, then  $cHALG_k$ is also enriched over $CoCoalg_k$.  In this section, we will obtain a second enrichment of commutative Hopf algebroids in cocommutative coalgebras, by using the shuffle product in Hochschild homology.

\smallskip
Let  $\mathcal U=(U,A_L,s_L,t_L,\Delta_L,\epsilon_L,S)$ be a commutative Hopf algebroid. Then,   $U$ and $A_L=A=A_R$ are commutative rings. 
We know from \cite[$\S$ 4.2]{Loday} that the Hochschild homology of a commutative algebra is equipped with a shuffle product structure. For a commutative Hopf algebroid
$\mathcal U=(U,A_L,s_L,t_L,\Delta_L,\epsilon_L,S)$, we now recall from \cite[$\S$ 4.4.1]{Ko3} the $(p,q)$-shuffle product
\begin{equation}
sh_{pq}(\mathcal U):C_p(\mathcal U)\otimes C_q(\mathcal U)\longrightarrow C_{p+q}(\mathcal U)
\end{equation} which is given by the formula (for $p$, $q\geq 1$)
\begin{equation}
sh_{pq}(\mathcal U)((u^1\otimes ...\otimes u^p)\otimes (u^{p+1}\otimes ...\otimes u^{p+q})):=\underset{\sigma\in Sh(p,q)}{\sum} sgn(\sigma)(u^{\sigma^{-1}(1)}\otimes ... \otimes u^{\sigma^{-1}(p+q)})
\end{equation} Here $Sh(p,q)$ is the set of $(p,q)$-shuffles, i.e.,
\begin{equation}
Sh(p,q):=\{\mbox{$\sigma\in S_{p+q}$ $\vert$ $\sigma(1)<...<\sigma(p)$; $\sigma(p+1)<...<\sigma(p+q)$}\}
\end{equation} For $p=q=0$, the shuffle product is given by setting $sh_{00}(\mathcal U)$ to be the multiplication on $A$. Further, one has (see \cite[$\S$ 4.4.1]{Ko3})
\begin{equation}
\begin{array}{c}
sh_{p0}(\mathcal U):C_p(\mathcal U)\otimes C_0(\mathcal U)\longrightarrow C_p(\mathcal U)\qquad (u^1\otimes ...\otimes u^p)\otimes a \mapsto (t_L(a)u^1\otimes ... \otimes u^p)\\
sh_{0q}(\mathcal U):C_0(\mathcal U)\otimes C_q(\mathcal U)\longrightarrow C_q(\mathcal U)\qquad a\otimes  (u^1\otimes ...\otimes u^p)\mapsto (u^1\otimes...\otimes  u^qt_L(a))\\
\end{array}
\end{equation} for $p\geq 1$ and $q\geq 1$. 
There is now an induced product structure 
$
sh_{pq}(\mathcal U):HH_p(\mathcal U)\otimes HH_q(\mathcal U)\longrightarrow HH_{p+q}(\mathcal U)
$ which makes the   the Hochschild homology $HH_\bullet(\mathcal U):=\underset{n\geq 0}{\bigoplus}HH_n(\mathcal U)$  of a commutative Hopf algebroid $\mathcal U$
into a graded algebra (see \cite[$\S$ 4.4.1]{Ko3}) that we denote by $(HH_\bullet(\mathcal U),sh(\mathcal U))$. 

\begin{thm}\label{P4.1n}
Let $\mathcal U$, $\mathcal U'$ be commutative Hopf algebroids. Let $C$ be a cocommutative coalgebra and let $(\Psi,\psi):C\longrightarrow V(\mathcal U,\mathcal U')$ be a measuring of Hopf algebroids. Then, the induced $K$-linear map
\begin{equation}
\underline{\Psi}^{hoc}:C\longrightarrow Vect_k(HH_\bullet(\mathcal U),HH_\bullet(\mathcal U'))\qquad x\mapsto (\underline{\Psi}^{hoc}_\bullet(x):HH_\bullet(\mathcal U)\longrightarrow HH_\bullet(\mathcal U'))
\end{equation} gives a measuring of algebras from $(HH_\bullet(\mathcal U),sh(\mathcal U))$ to $(HH_\bullet(\mathcal U'),sh(\mathcal U'))$.
\end{thm}

\begin{proof} The unit in $(HH_\bullet(\mathcal U),sh(\mathcal U))$ is given by the class of the unit $1_A\in A=C_0(\mathcal U)$. Since $\psi:C\longrightarrow Hom_K(A,A')$ gives in particular a measuring from $A$ to $A'$, we have $\underline{\Psi}^{hoc}_\bullet(x)(1_A)=\epsilon_C(x)1_{A'}$, where $\epsilon_C$ is the counit on $C$. We now note that for any $x\in  C$ and $p$, $q\geq 1$, we have
\begin{equation}
\begin{array}{l}
\underline{\Psi}_{p+q}(x)(sh_{pq}(\mathcal U)((u^1\otimes ...\otimes u^p)\otimes (u^{p+1}\otimes ...\otimes u^{p+q})))\\
=\underline{\Psi}_{p+q}(x)\left(\underset{\sigma\in Sh(p,q)}{\sum} sgn(\sigma)(u^{\sigma^{-1}(1)}\otimes ... \otimes u^{\sigma^{-1}(p+q)})\right)\\
=\underset{\sigma\in Sh(p,q)}{\sum} sgn(\sigma)(x_{(1)}(u^{\sigma^{-1}(1)})\otimes ... \otimes x_{(p+q)}(u^{\sigma^{-1}(p+q)}))\\
=\underset{\sigma\in Sh(p,q)}{\sum} sgn(\sigma)(x_{\sigma^{-1}(1)}(u^{\sigma^{-1}(1)})\otimes ... \otimes x_{\sigma^{-1}(p+q)}(u^{\sigma^{-1}(p+q)}))\qquad\qquad \mbox{(because
 $C$ is cocommutative)}\\
 =sh_{pq}(\mathcal U)((x_{(1)}(u^1)\otimes ...\otimes x_{(p)}(u^p))\otimes (x_{(p+1)}(u^{p+1})\otimes ...\otimes x_{(p+q)}(u^{p+q})))\\
\end{array}
\end{equation} 
For $p\geq 1$, we have
\begin{equation}
\begin{array}{ll}
\underline{\Psi}_{p}(x)(sh_{p0}(\mathcal U)((u^1\otimes ...\otimes u^p)\otimes a)&=
\underline{\Psi}_{p}(x)( (t_L(a)u^1\otimes ... \otimes u^p))\\
&= (x_{(1)}(t_L(a)u^1)\otimes ... \otimes x_{(p)}(u^p))\\
&=(x_{(1)}(t_L(a))x_{(2)}(u^1)\otimes ... \otimes x_{(p+1)}(u^p))\\
&=(t_L'(x_{(p+1)}(a)))x_{(1)}(u^1)\otimes ... \otimes x_{(p)}(u^p))\\
&=sh_{p0}(\mathcal U)((x_{(1)}(u^1)\otimes ...\otimes x_{(p)}(u^p))\otimes x_{(p+1)}(a))\\
\end{array}
\end{equation}
We can similarly verify the case for $sh_{0q}$ with $q\geq 1$ and for $sh_{00}$. 
This proves the result.

\end{proof}

Our next objective is to use Proposition \ref{P4.1n} to obtain an enrichment of commutative Hopf algebroids over the category of cocommutative coalgebras. For that we recall the following fact: if $R$, $R'$ are $k$-algebras, the category of coalgebra measurings from $R$ to $R'$ contains a final object $\varphi(R,R'):\mathcal M(R,R')\longrightarrow
Vect_k(R,R')$ (see Sweedler \cite{Sweed}). Then, 
$\mathcal M(R,R')$ is known as the universal measuring coalgebra. We let $\mathcal M_c(R,R')$ be the cocommutative part of the coalgebra $\mathcal M(R,R')$.  Then,  
the restriction $\varphi_c(R,R'):\mathcal M_c(R,R')\hookrightarrow \mathcal M(R,R')\longrightarrow Vect_k(R,R')$ becomes the final
object in the category of cocommutative coalgebra measurings from $R$ to $R'$ (see \cite[Proposition 1.4]{GM1}, \cite{GM2}).  Further, the objects $\mathcal M_c(R,R')$ give an enrichment of $k$-algebras over cocommutative $k$-coalgebras.

\smallskip
We now define the enriched category $\widetilde{cHALG}_k$ whose objects are commutative Hopf algebroids over $k$ and whose hom-objects are defined by setting
\begin{equation}\label{enr4}
\widetilde{cHALG}_k(\mathcal U,\mathcal U'):=\mathcal M_c((HH_\bullet(\mathcal U),sh(\mathcal U)),(HH_\bullet(\mathcal U'),sh(\mathcal U')))\in CoCoalg_k
\end{equation} for commutative Hopf algebroids $\mathcal U$, $\mathcal U'$. Since each $(HH_\bullet(\mathcal U),sh(\mathcal U))$ is an algebra, we also have a canonical morphism
$k\longrightarrow \mathcal M_c((HH_\bullet(\mathcal U),sh(\mathcal U)),(HH_\bullet(\mathcal U),sh(\mathcal U)))$ of cocommutative coalgebras. 

\begin{lem}\label{L4.2}  Let $\mathcal U$, $\mathcal U'$ be commutative Hopf algebroids. Then, there is a canonical morphism of cocommutative coalgebras
\begin{equation}\label{4.7yg}
\tau(\mathcal U,\mathcal U'):\mathcal M_c(\mathcal U,\mathcal U')\longrightarrow \mathcal M_c((HH_\bullet(\mathcal U),sh(\mathcal U)),(HH_\bullet(\mathcal U'),sh(\mathcal U')))
\end{equation}

\end{lem}

\begin{proof}
In the notation of Proposition \ref{P2.5}, $(\Phi,\phi):\mathcal M_c(\mathcal U,\mathcal U')\longrightarrow V(\mathcal U,\mathcal U')$ is a cocommutative measuring from $\mathcal U$ to $\mathcal U'$. By Proposition
\ref{P4.1n}, this induces a measuring  of algebras from $(HH_\bullet(\mathcal U),sh(\mathcal U))$ to $(HH_\bullet(\mathcal U'),sh(\mathcal U'))$. By the  property of the universal
cocommutative measuring coalgebra $\mathcal M_c((HH_\bullet(\mathcal U),sh(\mathcal U)),(HH_\bullet(\mathcal U'),sh(\mathcal U')))$, we now obtain an induced morphism
$\tau(\mathcal U,\mathcal U')$ as in \eqref{4.7yg}.
\end{proof}

\begin{Thm}\label{T4.3} There is a $CoCoalg_k$-enriched functor 
$
\tau: cHALG_k\longrightarrow \widetilde{cHALG}_k
$ which is identity on objects and whose mapping on hom-objects is given by
\begin{equation}
\tau(\mathcal U,\mathcal U'):cHALG_k(\mathcal U,\mathcal U')=\mathcal M_c(\mathcal U,\mathcal U')\longrightarrow \mathcal M_c((HH_\bullet(\mathcal U),sh(\mathcal U)),(HH_\bullet(\mathcal U'),sh(\mathcal U')))= \widetilde{cHALG}_k(\mathcal U,\mathcal U')
\end{equation} for commutative Hopf algebroids $\mathcal U$, $\mathcal U'$ over $k$. 
\end{Thm}

\begin{proof}
Let $\mathcal U$, $\mathcal U'$, $\mathcal U''$ be commutative Hopf algebroids. We  will show that the following diagram commutes
\begin{equation}\label{cd4.1j}
\begin{CD}
\mathcal  M_c(\mathcal U,\mathcal U')\otimes \mathcal  M_c(\mathcal U',\mathcal U'')@>\circ>>\mathcal  M_c(\mathcal U,\mathcal U'')\\
@V\tau(\mathcal U,\mathcal U')\otimes \tau(\mathcal U',\mathcal U'') VV @VV\tau(\mathcal U,\mathcal U'')V\\
\mathcal M_c(HH_\bullet(\mathcal U),HH_\bullet(\mathcal U'))\otimes \mathcal M_c(HH_\bullet(\mathcal U'),HH_\bullet(\mathcal U'')) @>\circ>> \mathcal M_c(HH_\bullet(\mathcal U),HH_\bullet(\mathcal U''))\\
\end{CD}
\end{equation} The top horizontal composition $\circ: \mathcal  M_c(\mathcal U,\mathcal U')\otimes \mathcal  M_c(\mathcal U',\mathcal U'')\longrightarrow \mathcal  M_c(\mathcal U,\mathcal U'')$ in \eqref{cd4.1j} is obtained from Theorem \ref{T2.7}, while the bottom horizontal composition $\circ: \mathcal M_c(HH_\bullet(\mathcal U),HH_\bullet(\mathcal U'))\otimes \mathcal M_c(HH_\bullet(\mathcal U'),HH_\bullet(\mathcal U'')) \longrightarrow \mathcal M_c(HH_\bullet(\mathcal U),HH_\bullet(\mathcal U''))$ is obtained from the enrichment of algebras in cocommutative coalgebras. 

\smallskip
From Lemma \ref{L4.2} and Theorem \ref{T2.7}, we note that all the maps in \eqref{cd4.1j} are morphisms of cocommutative coalgebras. It follows from the   property of the universal cocommutative measuring coalgebra $\mathcal M_c(HH_\bullet(\mathcal U),HH_\bullet(\mathcal U''))$ that in order to show that \eqref{cd4.1j} commutes, it suffices to verify that the following two compositions are equal
\begin{equation}\label{cd4.12j}
\begin{array}{ccc}
\begin{CD}
\mathcal  M_c(\mathcal U,\mathcal U')\otimes \mathcal  M_c(\mathcal U',\mathcal U'')\\
@V\tau(\mathcal U,\mathcal U')\otimes \tau(\mathcal U',\mathcal U'') VV \\
\mathcal M_c(HH_\bullet(\mathcal U),HH_\bullet(\mathcal U'))\otimes \mathcal M_c(HH_\bullet(\mathcal U'),HH_\bullet(\mathcal U''))\\
@V\circ VV\\
\mathcal M_c(HH_\bullet(\mathcal U),HH_\bullet(\mathcal U''))\\
@VV\varphi_c(HH_\bullet(\mathcal U),HH_\bullet(\mathcal U''))V\\
Vect_k(HH_\bullet(\mathcal U),HH_\bullet(\mathcal U''))\\
\end{CD} & \qquad & 
\begin{CD}
\mathcal  M_c(\mathcal U,\mathcal U')\otimes \mathcal  M_c(\mathcal U',\mathcal U'')\\
@VV\circ V\\
\mathcal  M_c(\mathcal U,\mathcal U'')\\
@VV\tau(\mathcal U,\mathcal U'')V\\
\mathcal M_c(HH_\bullet(\mathcal U),HH_\bullet(\mathcal U''))\\
@VV\varphi_c(HH_\bullet(\mathcal U),HH_\bullet(\mathcal U''))V\\
Vect_k(HH_\bullet(\mathcal U),HH_\bullet(\mathcal U''))\\
\end{CD}\\
\end{array}
\end{equation} For the sake of convenience, we denote the left vertical composition in \eqref{cd4.12j} by $\psi_1$ and the right vertical composition by $\psi_2$.  We now consider 
$x\in \mathcal  M_c(\mathcal U,\mathcal U')$, $y\in \mathcal  M_c(\mathcal U',\mathcal U'')$ and $(u^1\otimes ...\otimes u^p)\in C_p(\mathcal U)$. We see that
\begin{equation}\label{413ar}
\begin{array}{ll}
\psi_2(x\otimes y)(u^1\otimes ...\otimes u^p)&=(y\circ x)(u^1\otimes ...\otimes u^p)\\
&=(y\circ x)_{(1)}(u^1)\otimes ...\otimes (y\circ x)_{(p)}(u^p)\\
\end{array}
\end{equation} Since $\circ: \mathcal  M_c(\mathcal U,\mathcal U')\otimes \mathcal  M_c(\mathcal U',\mathcal U'')\longrightarrow \mathcal  M_c(\mathcal U,\mathcal U'')$ is a morphism of coalgebras, we note that $(y\circ x)_{(1)}\otimes ...\otimes (y\circ x)_{(p)}=(y_{(1)}\circ x_{(1)})\otimes ...\otimes (y_{(p)}\circ x_{(p)})$. Combining with \eqref{413ar}, we see that the right vertical composition in \eqref{cd4.12j} may be described explicitly as
\begin{equation}\label{414cg}
\psi_2(x\otimes y)(u^1\otimes ...\otimes u^p)=(y_{(1)}\circ x_{(1)})(u^1)\otimes ...\otimes (y_{(p)}\circ x_{(p)})(u^p)=y_{(1)}(x_{(1)}(u^1))\otimes ...\otimes y_{(p)}(x_{(p)}(u^p))
\end{equation} On the other hand, we note that the following diagram is commutative
\begin{equation}\label{415qs}
\begin{CD}
\mathcal  M_c(\mathcal U,\mathcal U')\otimes \mathcal  M_c(\mathcal U',\mathcal U'')@>\circ(\tau(\mathcal U,\mathcal U')\otimes \tau(\mathcal U',\mathcal U''))>>\mathcal M_c(HH_\bullet(\mathcal U),HH_\bullet(\mathcal U'')) \\
@V(\varphi_c(HH_\bullet(\mathcal U),HH_\bullet(\mathcal U'))\circ \tau(\mathcal U,\mathcal U'))\otimes V(\varphi_c(HH_\bullet(\mathcal U'),HH_\bullet(\mathcal U''))\circ \tau(\mathcal U',\mathcal U''))V @V\varphi_c(HH_\bullet(\mathcal U),HH_\bullet(\mathcal U''))VV\\
Vect_k(HH_\bullet(\mathcal U),HH_\bullet(\mathcal U'))\otimes Vect_k(HH_\bullet(\mathcal U'),HH_\bullet(\mathcal U''))@>\circ >> Vect_k(HH_\bullet(\mathcal U),HH_\bullet(\mathcal U''))\\
\end{CD}
\end{equation} From \eqref{415qs}, it follows that the left vertical composition in \eqref{cd4.12j} may be described explicitly as
\begin{equation}\label{416tv}
\begin{array}{ll}
\psi_1(x\otimes y)(u^1\otimes ...\otimes u^p)&=y(x(u^1\otimes ...\otimes u^p))\\
&=y_{(1)}(x_{(1)}(u^1))\otimes ...\otimes y_{(p)}(x_{(p)}(u^p))\\
\end{array}
\end{equation} From \eqref{414cg} and \eqref{416tv}, we see that $\psi_1=\psi_2$ and hence the diagram \eqref{cd4.1j} commutes. Similarly by considering the coalgebra $k$ and using the fact that its $p$-th iterated coproduct $\Delta^p(1)=1\otimes ...\otimes 1 \textrm{}\mbox{($p$-times)}
$, we see that the following compositions are equal
\begin{equation} \label{417ge}
\begin{array}{c}k\longrightarrow \mathcal M_c(HH_\bullet(\mathcal U),HH_\bullet(\mathcal U))\xrightarrow{\varphi_c(HH_\bullet(\mathcal U),HH_\bullet(\mathcal U))} Vect_k(HH_\bullet(\mathcal U),HH_\bullet(\mathcal U))
\\ k\longrightarrow \mathcal M_c(\mathcal U,\mathcal U)\xrightarrow{\tau(\mathcal U,\mathcal U)} \mathcal M_c(HH_\bullet(\mathcal U),HH_\bullet(\mathcal U))\xrightarrow{\varphi_c(HH_\bullet(\mathcal U),HH_\bullet(\mathcal U))} Vect_k(HH_\bullet(\mathcal U),HH_\bullet(\mathcal U))
\end{array}
\end{equation} It follows from \eqref{417ge} that the following diagram commutes
\begin{equation}
\xymatrix{k \ar[rr] \ar[rd]&  &\mathcal M_c(HH_\bullet(\mathcal U),HH_\bullet(\mathcal U))\\
& \mathcal M_c(\mathcal U,\mathcal U) \ar[ur]_{\tau(\mathcal U,\mathcal U)} & \\
}
\end{equation}
This proves the result.
\end{proof}
\section{Comodule measurings for SAYD modules}

Let  $\mathcal U=(U,A_L,s_L,t_L,\Delta_L,\epsilon_L,S)$ be a Hopf algebroid. From now onwards, we set $A^e:=A\otimes_kA^{op}$ and define
\begin{equation}\label{5v}
\eta_L: A^e=A\otimes_kA^{op}\xrightarrow{s_L\otimes t_L}U\otimes U \longrightarrow U
\end{equation} where the second arrow in \eqref{5v} is the multiplication on $U$.  Following \cite[$\S$ 2]{KoKr}, we note that there are now four commuting actions of
$A$ on $U$ which are denoted as follows
\begin{equation}\label{whiteblack}
a\triangleright u\triangleleft b:=s_L(a)t_L(b)u\qquad a\blacktriangleright u\blacktriangleleft b:=us_L(b)t_L(a) \qquad a,b\in A,\textrm{ }u\in U
\end{equation} By Definition \ref{D2.1}, we then have an $A$-coring
\begin{equation}
\Delta_L:U\longrightarrow U_\triangleleft\otimes_A {_\triangleright}U \qquad \epsilon_L: U\longrightarrow A
\end{equation} The left action $\blacktriangleright$ of $A$ on $U$ may be treated as a right action of $A^{op}$ on $U$. Similarly, the right action $\triangleleft$ of
$A$ on $U$ may be treated as a left action by $A^{op}$. Accordingly, we may consider the tensor product
\begin{equation}\label{54ze}
{_\blacktriangleright}U\otimes_{A^{op}}U_{\triangleleft}:=U\otimes_kU/\mbox{span$\{\mbox{$a\blacktriangleright u\otimes v-u\otimes v\triangleleft a$ $\vert$ $u$, $v\in U$, $a\in A$} \}$}
\end{equation} There is now a Hopf-Galois map (see \cite{BoK}, \cite{KoKr}, \cite{Sch2})
\begin{equation}\label{hg5.5}
\beta(\mathcal U):{_\blacktriangleright}U\otimes_{A^{op}}U_{\triangleleft}\longrightarrow U_\triangleleft \otimes_A {_\triangleright}U \qquad u\otimes_{A^{op}}v\mapsto u_{(1)}\otimes_Au_{(2)}v
\end{equation} Since $\mathcal U$ is a Hopf algebroid, it follows (see \cite[Proposition 4.2]{BoK}) that the morphism $\beta(\mathcal U)$ in \eqref{hg5.5} is a bijection. Accordingly, in the notation of \cite{KoKr}, \cite{Sch2}, we write 
\begin{equation}\label{sty5.6}
u_+\otimes_{A^{op}}u_-:=\beta(\mathcal U)^{-1}(u\otimes_A1) \qquad u\in U
\end{equation} In this section, we will consider comodule measurings between stable anti-Yetter Drinfeld modules over Hopf algebroids. For this, we first recall the notion of comodule measuring between ordinary modules. Let $R$, $R'$ be $k$-algebras and let $P$, $P'$ be right modules over $R$ and $R'$ respectively. Then, a comodule measuring from $P$ to $P'$ consists of a pair of maps (see \cite{Bat}, \cite{V1})
\begin{equation}
\psi: C\longrightarrow Vect_k(R,R')\qquad \omega: D\longrightarrow Vect_k(P,P')
\end{equation} where $C$ is a $k$-coalgebra, $D$ is a right $C$-comodule, $\psi:C\longrightarrow Vect_k(R,R')$  is a coalgebra measuring and
\begin{equation}
\omega(y)(pr)=y(pr)=y_{(0)}(p)y_{(1)}(r)=\omega(y_{(0)})(p)\psi(y_{(1)})(r)
\end{equation} for $y\in D$, $p\in P$ and $r\in R$.    For $\mathcal U=(U,A_L,s_L,t_L,\Delta_L,\epsilon_L,S)$, we will now recall the notions of $\mathcal U$-modules, $\mathcal U$-comodules and  stable anti-Yetter Drinfeld modules.

\begin{defn}\label{D5.1} (see \cite[$\S$ 2.4]{KoKr})
Let   $\mathcal U=(U,A_L,s_L,t_L,\Delta_L,\epsilon_L,S)$  be a Hopf algebroid. A right $\mathcal U$-module $P$ is  a right module over the $k$-algebra $U$.  Because of the ring
homomorphism $\eta_L:A^e
\longrightarrow U$, any right $\mathcal U$-module $P$ is also equipped with a  right $A^e$-module structure (or $(A,A)$-bimodule structure) given by
\begin{equation}\label{e5.09u}
b\blacktriangleright p\blacktriangleleft a=p(a\otimes b)=p\eta_L((a\otimes 1)(1\otimes b))=ps_L(a)t_L(b)
\end{equation} for $(a\otimes b)\in A^e=A\otimes_kA^{op}$ and $p\in P$. 
\end{defn}

\begin{defn}\label{D5.2} (see \cite{B1}, \cite{BrzWi}, \cite{KoKr}, \cite{Sch1})
Let   $\mathcal U=(U,A_L,s_L,t_L,\Delta_L,\epsilon_L,S)$  be a Hopf algebroid. A left $\mathcal U$-comodule $P$ is a left comodule over the $A$-coring $(U,\Delta_L:U\longrightarrow U\otimes_{A_L}U,\epsilon_L:U\longrightarrow A_L)$. In particular, a left $\mathcal U$-comodule $P$  is equipped with a left $A$-module structure $(a,p)\mapsto ap$ as well as a left $A$-module map
\begin{equation}
\Delta_P:P\longrightarrow U_\triangleleft\otimes_A P \qquad p\mapsto p_{(-1)}\otimes p_{(0)}
\end{equation}
\end{defn}

Following \cite[$\S$ 2.5]{KoKr}, we note that any left $\mathcal U$-comodule $P$ also carries a right $A$-module structure given by setting
\begin{equation}\label{r5}
pa:=\epsilon_L(p_{(-1)}s_L(a))p_{(0)} 
\end{equation} for $p\in P$, $a\in A$. This makes  any left $\mathcal U$-comodule $P$ into a right $A^e=A\otimes_kA^{op}$-module by setting 
\begin{equation}\label{e5.9u}
p(a\otimes b)=bpa=b\epsilon_L(p_{(-1)}s_L(a))p_{(0)} 
\end{equation} for $p\in P$ and  $(a\otimes b)\in A^e$. 

\begin{defn}\label{D5.3} (see \cite[Definition 2.7]{KoKr})  Let   $\mathcal U=(U,A_L,s_L,t_L,\Delta_L,\epsilon_L,S)$  be a Hopf algebroid.  An anti-Yetter Drinfeld module $P$ over
$\mathcal U$ is given by the following data

\smallskip
(1) A right $\mathcal U$-module structure on $P$ denoted by $(p,u)\mapsto pu$ for $p\in P$ and $u\in U$. 

\smallskip
(2) A left $\mathcal U$-comodule structure on $P$ which gives $\Delta_P:P\longrightarrow U_\triangleleft\otimes_A P$.

\smallskip
(3) The right $A^e$-module structure on $P$ induced by \eqref{e5.09u} coincides with the right $A^e$-module structure on $P$ as in \eqref{e5.9u}:
\begin{equation}\label{ditto}
ps_L(a)t_L(b)=b\blacktriangleright p\blacktriangleleft a=b\epsilon_L(p_{(-1)}s_L(a))p_{(0)} 
\end{equation}

\smallskip
(4) For $u\in U$ and $p\in P$, one has
\begin{equation}
\Delta_P(pu)=u_-p_{(-1)}u_{+(1)}\otimes_Ap_{(0)}u_{+(2)}
\end{equation}

\smallskip A morphism of anti-Yetter Drinfeld modules over
$\mathcal U$   is a morphism of vector spaces that is compatible with the right $\mathcal U$-module and left $\mathcal U$-comodule structures. The category of anti-Yetter Drinfeld modules will be denoted by ${^{\mathcal U}}aYD_{\mathcal U}$. 
Additionally, an anti-Yetter Drinfeld module $P$ is said to be stable, i.e., an SAYD module, if  for any $p\in P$, one has $p_{(0)}p_{(-1)}=p$. 
\end{defn}

\begin{lem}\label{L5.4b}
Let $R$, $R'$ be $k$-algebras and let $R^e=R\otimes_kR^{op}$, $R'^e=R'\otimes_kR'^{op}$ be their respective enveloping algebras. Let $C$ be a cocommutative $k$-coalgebra and let $\psi:C\longrightarrow Vect_k(R,R')$ be a measuring. Then, 
\begin{equation}\label{515t}
\psi^e:C\longrightarrow Vect_k(R^e,R'^e)\qquad \psi^e(x)(r_1\otimes r_2)=x(r_1\otimes r_2)=x_{(1)}(r_1)\otimes x_{(2)}(r_2)=\psi(x_{(1)})(r_1)\otimes \psi(x_{(2)})(r_2)
\end{equation} is a measuring of algebras.
\end{lem}

\begin{proof}
From \eqref{515t}, it is immediate that $x(1\otimes 1)=\epsilon_C(x)(1\otimes 1)$, where $\epsilon_C$ is the counit on $C$. Since $C$ is cocommutative, we have for $(r_1\otimes r_2)$, $(r_3\otimes r_4)\in 
R^e$
\begin{equation}
\begin{array}{ll}
x((r_1\otimes r_2)(r_3\otimes r_4))= x(r_1r_3\otimes r_4r_2)&=x_{(1)}(r_1r_3)\otimes x_{(2)}(r_4r_2)\\
&=x_{(1)}(r_1)x_{(2)}(r_3)\otimes x_{(3)}(r_4)x_{(4)}(r_2)\\
&=x_{(1)}(r_1)x_{(3)}(r_3)\otimes  x_{(4)}(r_4)x_{(2)}(r_2)\\
&=(x_{(1)}(r_1)\otimes x_{(2)}(r_2))(x_{(3)}(r_3)\otimes x_{(4)}(r_4))\\
&=x_{(1)}(r_1\otimes r_2)x_{(2)}(r_3\otimes r_4)\\
\end{array}
\end{equation}
\end{proof}

\begin{lem}\label{L5.5vq}
Let  $\mathcal U =(U,A_L,s_L,t_L,\Delta_L,\epsilon_L,S)$ and
$\mathcal U' =(U',A_L',s'_L,t'_L,\Delta'_L,\epsilon'_L,S')$ be Hopf algebroids over $k$. Let $P$ (resp. $P'$) be an SAYD-module over $\mathcal U$ (resp. $\mathcal U'$). Let $C$ be a cocommutative $k$-coalgebra and $D$ be a right $C$-comodule. Suppose that we are given the following data
\begin{equation} \Psi:C\longrightarrow Vect_k(U,U')\qquad \psi:C\longrightarrow Vect_k(A,A')\qquad \Omega: D\longrightarrow Vect_k(P,P')
\end{equation} such that

\smallskip
(1) $(\Psi,\psi)$ is a measuring of Hopf algebroids from $\mathcal U$ to $\mathcal U'$.

\smallskip
(2) $(\Psi,\Omega)$ is a comodule measuring from the right $U$-module $P$ to the right $U'$ module $P'$.

\smallskip
Then, we have:

\smallskip
(a) $(\psi^e,\Omega)$ is a comodule measuring from the right $A^e$-module $P$ to the right $A'^e$-module $P'$. 

\smallskip
(b) For each $y\in D$, the following morphism is well-defined
\begin{equation}\label{ra5.18}
y:U_\triangleleft\otimes_AP\longrightarrow U'_\triangleleft \otimes_{A'}P'\qquad y(u\otimes_A p):=\Psi(y_{(1)})(u)\otimes_{A'} \Omega(y_{(0)})(p)=y_{(1)}(u)\otimes_{A'}y_{(0)}(p)
\end{equation}
\end{lem}
\begin{proof}(a) Since $C$ is cocommutative, we already know from Lemma \ref{L5.4b}  that $\psi^e:C\longrightarrow Vect_k(A^e,A'^e)$ is a coalgebra measuring from
$A^e$ to $A'^e$. We now consider $(a\otimes b)\in A^e=A\otimes_kA^{op}$ and $p\in P$. By \eqref{e5.09u}, we know that $p(a\otimes b)=ps_L(a)t_L(b)$. For any $y\in D$, we now have
\begin{equation}
\begin{array}{ll}
\Omega(y)(p(a\otimes b))=\Omega(y)(ps_L(a)t_L(b))&=\Omega(y_{(0)})(p)\Psi(y_{(1)})(s_L(a)t_L(b))\\
&=\Omega(y_{(0)})(p)\Psi(y_{(1)})(s_L(a))\Psi(y_{(2)})(t_L(b))\\
&=\Omega(y_{(0)})(p)s_L'(\psi(y_{(1)})(a))t_L'(\psi(y_{(2)})(b))\\
&=\Omega(y_{(0)})(p)(\psi(y_{(1)})(a)\otimes \psi(y_{(2)})(b))\\
&=\Omega(y_{(0)})(p)(\psi^e(y_{(1)})(a\otimes b))\\
\end{array}
\end{equation}

\smallskip
(b) Since $P$ and $P'$ are SAYD modules, it follows from the definitions in \eqref{e5.09u}, \eqref{e5.9u} and the condition in \eqref{ditto} that
\begin{equation}\label{apap}
ap=pt_L(a) \qquad a'p'=p't_L'(a') \qquad a\in A, a'\in A', p\in P, p'\in P'
\end{equation} where the left hand side of the equalities in \eqref{apap} comes from the left $A$-module action on $P$ (resp. the left $A'$-module action on $P'$) appearing in the structure
map $\Delta_P:P\longrightarrow U_\triangleleft\otimes_A P$ (resp. the structure map $\Delta'_{P'}:P'\longrightarrow U'_\triangleleft\otimes_{A'} P'$). For $a\in A$, $u\in U$ and 
$p\in P$, we now see that
\begin{equation}\label{5.21x}
\begin{array}{lll}
y(u\otimes_Aap)&=\Psi(y_{(1)})(u)\otimes_{A'} \Omega(y_{(0)})(ap)&\\
&=\Psi(y_{(1)})(u)\otimes_{A'} \Omega(y_{(0)})(pt_L(a))&\mbox{(using \eqref{apap})}\\
&=\Psi(y_{(2)})(u)\otimes_{A'} \Omega(y_{(0)})(p)\Psi(y_{(1)})(t_L(a))& \\
&=\Psi(y_{(2)})(u)\otimes_{A'} \Omega(y_{(0)})(p)t_L'(\psi(y_{(1)})(a) )& \\
&=\Psi(y_{(2)})(u)\otimes_{A'}\psi(y_{(1)})(a)  \Omega(y_{(0)})(p)&\mbox{(using \eqref{apap})} \\
&=\Psi(y_{(2)})(u)\triangleleft \psi(y_{(1)})(a)\otimes_{A'}  \Omega(y_{(0)})(p)&\\
&=t_L'(\psi(y_{(1)})(a))\Psi(y_{(2)})(u)\otimes_{A'}  \Omega(y_{(0)})(p)&\mbox{(using \eqref{whiteblack})} \\
\end{array}
\end{equation} On the other hand, we also have
\begin{equation}\label{5.22x}
\begin{array}{lll}
y(u\triangleleft a\otimes_Ap)&=y(t_L(a)u\otimes_Ap)&\mbox{(using \eqref{whiteblack})} \\
&=\Psi(y_{(1)})(t_L(a)u)\otimes_{A'} \Omega(y_{(0)})(p)&\\
&=\Psi(y_{(1)})(t_L(a))\Psi(y_{(2)})(u)\otimes_{A'} \Omega(y_{(0)})(p)&\\
&=t_L'(\psi(y_{(1)})(a))\Psi(y_{(2)})(u)\otimes_{A'}  \Omega(y_{(0)})(p)&\\
\end{array}
\end{equation} This proves the result.
\end{proof}

We are now ready to introduce the notion of a comodule measuring between SAYD modules.

\begin{defn}\label{D5.6}
Let  $\mathcal U=(U,A_L,s_L,t_L,\Delta_L,\epsilon_L,S)$ and
$\mathcal U'=(U',A_L',s'_L,t'_L,\Delta'_L,\epsilon'_L,S')$ be Hopf algebroids over $k$. Let $P$ (resp. $P'$) be an SAYD-module over $\mathcal U$ (resp. $\mathcal U'$). 
Let $C$ be a cocommutative coalgebra. Then, a  $(C,\Psi,\psi)$-comodule measuring from $P$ to $P'$ consists of the following data
\begin{equation} \Psi:C\longrightarrow Vect_k(U,U')\qquad \psi:C\longrightarrow Vect_k(A,A')\qquad \Omega: D\longrightarrow Vect_k(P,P')
\end{equation} such that

\smallskip
(1) $(\Psi,\psi)$ is a measuring of Hopf algebroids from $\mathcal U$ to $\mathcal U'$.

\smallskip
(2) $(\Psi,\Omega)$ is a comodule measuring from the right $U$-module $P$ to the right $U'$ module $P'$.

\smallskip
(3) For each $y\in D$, the following diagram commutes
\begin{equation}\label{cd523}
\begin{CD}
P @>\Delta_P>> U_\triangleleft\otimes_A P\\
@Vy:=\Omega(y)VV @VyVV\\
P'@>\Delta'_{P'}>> U'_\triangleleft\otimes_{A'} P'\\
\end{CD}
\end{equation}
where the right vertical morphism is as defined in \eqref{ra5.18}
\end{defn} 

We will now construct universal measuring comodules between SAYD modules. By definition, the right comodules over a $k$-coalgebra $C$ are coalgebras over the comonad $\_\_\otimes_kC:Vect_k\longrightarrow Vect_k$. Accordingly, the forgetful functor $Comod-C\longrightarrow Vect_k$ from the category of right $C$-comodules has a right adjoint (see, for instance, 
\cite[$\S$ 2.4]{BBW}) that we denote by $\mathfrak R_C$, i.e., we have natural isomorphisms
\begin{equation}\label{radj5t}
Vect_k(D,V)\cong Comod-C(D,\mathfrak R_C(V))
\end{equation} for any $D\in Comod-C$ and $V\in Vect_k$.

\begin{Thm}\label{T5.7}
Let  $\mathcal U=(U,A_L,s_L,t_L,\Delta_L,\epsilon_L,S)$ and
$\mathcal U'=(U',A_L',s'_L,t'_L,\Delta'_L,\epsilon'_L,S')$ be Hopf algebroids over $k$. Let $P$ (resp. $P'$) be an SAYD-module over $\mathcal U$ (resp. $\mathcal U'$). Let $C$ be a cocommutative coalgebra and $(\Psi,\psi):C\longrightarrow V(\mathcal U,\mathcal U')$ be a measuring of Hopf algebroids. 

\smallskip Then, there exists a $(C,\Psi,\psi)$-comodule measuring 
$ \Theta:\mathcal Q_C(P,P')\longrightarrow Vect_k(P,P')$ satisfying the following property: given any $(C,\Psi,\psi)$-comodule measuring $\Omega:D\longrightarrow
Vect_k(P,P')$ from $P$ to $P'$, there exists a morphism $\chi:D\longrightarrow \mathcal Q_C(P,P')$ of right $C$-comodules such that the following diagram is commutative
\begin{equation}
\xymatrix{
\mathcal Q_C(P,P')\ar[rr]^{\Theta} && Vect_k(P,P')\\
&D\ar[ul]^\chi\ar[ur]_{\Omega}&\\
}
\end{equation}
\end{Thm}

\begin{proof}
We put $V:=Vect_k(P,P')$. By the adjunction in \eqref{radj5t}, there is a canonical morphism $\rho(V):\mathfrak R_C(V)\longrightarrow V$ of vector spaces. We set $\mathcal Q_C(P,P'):=
\sum Q$, where the sum is taken over all right $C$-subcomodules of $\mathfrak R_C(V)$ such that the restriction $\rho(V)|Q:Q\longrightarrow V=Vect_k(P,P')$ is a $(C,\Psi,\psi)$-comodule measuring from $P$ to $P'$ in the sense of Definition \ref{D5.6}. It is clear that $\Theta:\rho(V)|\mathcal Q_C(P,P'): \mathcal Q_C(P,P')\longrightarrow V=Vect_k(P,P')$ is a $(C,\Psi,\psi)$-comodule measuring. 

\smallskip
Additionally, given a  $(C,\Psi,\psi)$-comodule measuring $ \Omega:D\longrightarrow
Vect_k(P,P')$ from $P$ to $P'$, the adjunction in \eqref{radj5t} gives a morphism $\chi:D\longrightarrow \mathfrak R_C(V)$ of right $C$-comodules satisfying $\rho(V)\circ \chi=\Omega$. But then we notice that $\rho(V)|\chi(D):\chi(D)\longrightarrow 
V$ is a $(C,\Psi,\psi)$-comodule measuring, whence it follows that the image $\chi(D)\subseteq \mathcal Q_C(P,P')$. The result is now clear.
\end{proof}

\begin{lem}
\label{Lem5.8} Let  $\mathcal U=(U,A_L,s_L,t_L,\Delta_L,\epsilon_L,S)$, 
$\mathcal U'=(U',A_L',s'_L,t'_L,\Delta'_L,\epsilon'_L,S')$ and  $\mathcal U''=(U'',A''_L,s''_L,t''_L,\Delta''_L,\epsilon''_L,S'')$ be Hopf algebroids over $k$.  Let $P$, $P'$ and $P''$ be SAYD modules over $\mathcal U$, $\mathcal U'$ and $\mathcal U''$ respectively.  Suppose that we have:

\smallskip
(1) $ \Psi:C\longrightarrow Vect_k(U,U')$, $ \psi:C\longrightarrow Vect_k(A,A')$ and $ \Omega: D\longrightarrow Vect_k(P,P')$ giving the data of a comodule measuring  from
$P$ to $P'$.

\smallskip
(2) $ \Psi':C'\longrightarrow Vect_k(U',U'')$, $ \psi':C'\longrightarrow Vect_k(A',A'')$ and $ \Omega': D'\longrightarrow Vect_k(P',P'')$ giving the data of a comodule measuring  from
$P'$ to $P''$.

\smallskip
Then, the following
\begin{equation}\label{xcomp527}
\begin{array}{c}
(\Psi',\psi')\circ (\Psi,\psi): C\otimes C'\xrightarrow{(\Psi,\psi)\otimes (\Psi',\psi')}V(\mathcal U,\mathcal U') \otimes V(\mathcal U',\mathcal U'')\overset{\circ}{\longrightarrow}V(\mathcal U,\mathcal U'')  \\
\Omega'\circ \Omega: D\otimes D'\xrightarrow{\Omega\otimes \Omega'}Vect_k(P,P')\otimes Vect_k(P',P'')\overset{\circ}{\longrightarrow}Vect_k(P,P'')\\
\end{array}
\end{equation} 
gives the data of a comodule measuring  from $P$ to $P''$. There is also a canonical morphism of right $(C\otimes C')$-comodules
\begin{equation}\label{528k}
\mathcal Q_C(P,P')\otimes \mathcal Q_{C'}(P',P'')
\longrightarrow \mathcal Q_{C\otimes C'}(P,P'')
\end{equation}
\end{lem}
\begin{proof}
We know from Proposition \ref{P2.6} that $(\Psi',\psi')\circ (\Psi,\psi): C\otimes C'\longrightarrow V(\mathcal U,\mathcal U'') $ is a measuring of Hopf algebroids. It may also be directly verified that  $((\Psi',\psi')\circ (\Psi,\psi),\Omega'\circ \Omega)$ is a comodule measuring from the right $U$-module $P$ to the right $U''$-module $P''$. To check the condition \eqref{cd523} in Definition \ref{D5.6}, we observe that for any $y\otimes y'\in D\otimes D'$, $u\in U$ and $p\in P$:
\begin{equation}
(y\otimes y')(u\otimes_Ap)=(y\otimes y')_{(1)}(u)\otimes_{A''}(y\otimes y')_{(0)}(p)=y'_{(1)}(y_{(1)}(u))\otimes_{A''} y'_{(0)}(y_{(0)}(p))=y'(y(u\otimes_Ap)))
\end{equation} Since the measurings $(\Psi,\psi,\Omega)$ and $(\Psi',\psi',\Omega')$ both satisfy the condition in \eqref{cd523}, it is now clear that so does $((\Psi',\psi')\circ (\Psi,\psi),\Omega'\circ \Omega)$. Hence, \eqref{xcomp527} gives the data of a  comodule measuring from $P$ to $P''$. By definition, $\mathcal Q_C(P,P')$ (resp. $ \mathcal Q_{C'}(P',P'')$)  is a measuring comodule
from $P$ to $P'$ (resp. from $P'$ to $P''$). From \eqref{xcomp527} it now follows that $\mathcal Q_C(P,P')\otimes \mathcal Q_{C'}(P',P'')$    is a measuring comodule
from $P$ to $P''$. The morphism in \eqref{528k} is now obtained by the universal property of $\mathcal Q_{C\otimes C'}(P,P'')$. 
\end{proof}

We now consider the ``global category of comodules'' $Comod_k$ whose objects are pairs $(C,D)$, where $C$ is a cocommutative $k$-coalgebra and $D$ is a $C$-comodule. A morphism $(f,g):(C,D)\longrightarrow (C',D')$ in $Comod_k$ consists of a $k$-coalgebra morphism $f:C\longrightarrow C'$ and a morphism $g:D\longrightarrow D'$ of $C'$-comodules, where $D$ is treated as a $C'$-comodule by corestriction of scalars. It is clear that putting $(C,D)\otimes (C',D'):=(C\otimes C',D\otimes D')$ makes $Comod_k$ into a symmetric monoidal category.

\begin{Thm}\label{T5.9hh}
Let  $SAYD_k$ be the category given by:

\smallskip
(a) Objects: pairs $(\mathcal U,P)$, where $\mathcal U$ is a Hopf-algebroid and $P$ is an $SAYD$-module over $\mathcal U$

\smallskip
(b) Hom-objects: for pairs $(\mathcal U,P)$, $(\mathcal U',P')\in SAYD_k$, we set
\begin{equation}
SAYD_k((\mathcal U,P),(\mathcal U',P')):=(\mathcal M_c(\mathcal U,\mathcal U'), \mathcal Q_{\mathcal M_c(\mathcal U,\mathcal U')}(P,P'))\in Comod_k
\end{equation} Then, $SAYD_k$ is enriched over the symmetric monoidal category $Comod_k$.
\end{Thm}

\begin{proof} For any $(\mathcal U,P)\in SAYD_k$,  we know from the proof of Theorem \ref{T2.7} that there is a morphism $k\longrightarrow \mathcal M_c(\mathcal U,\mathcal U)$  of $k$-coalgebras. Using scalar multiples of the identity map and applying  the universal property in Theorem \ref{T5.7}, we obtain a  morphism $k\longrightarrow  \mathcal Q_{\mathcal M_c(\mathcal U,\mathcal U)}(P,P)$. We now  consider $(\mathcal U,P)$, $(\mathcal U',P')$,  $(\mathcal U'',P'')\in SAYD_k$. Applying Lemma \ref{Lem5.8} with $C=\mathcal M_c(\mathcal U,\mathcal U')$ and
$C'=\mathcal M_c(\mathcal U',\mathcal U'')$,  we obtain a morphism $\mathcal Q_{\mathcal M_c(\mathcal U,\mathcal U')}(P,P')\otimes \mathcal Q_{\mathcal M_c(\mathcal U',\mathcal U'')}(P',P'')\longrightarrow \mathcal Q_{\mathcal M_c(\mathcal U,\mathcal U')\otimes \mathcal M_c(\mathcal U',\mathcal U'')}(P,P'')$ of $(\mathcal M_c(\mathcal U,\mathcal U')\otimes \mathcal M_c(\mathcal U',\mathcal U''))$-comodules. From the proof of Theorem \ref{T2.7}, we already have a morphism $\mathcal M_c(\mathcal U,\mathcal U')\otimes \mathcal M_c(\mathcal U',\mathcal U'')\longrightarrow \mathcal M_c(\mathcal U,\mathcal U'')$ of $k$-coalgebras. Combining, we have a morphism in $Comod_k$
\begin{equation}\label{ar531}
SAYD_k((\mathcal U,P),(\mathcal U',P'))\otimes SAYD_k((\mathcal U',P'),(\mathcal U'',P''))\longrightarrow (\mathcal M_c(\mathcal U,\mathcal U''), \mathcal Q_{\mathcal M_c(\mathcal U,\mathcal U')\otimes \mathcal M_c(\mathcal U',\mathcal U'')}(P,P''))
\end{equation}  In \eqref{ar531}, $ \mathcal Q_{\mathcal M_c(\mathcal U,\mathcal U')\otimes \mathcal M_c(\mathcal U',\mathcal U'')}(P,P'')$ becomes a $\mathcal M_c(\mathcal U,\mathcal U'')$-comodule via the morphism  $\mathcal M_c(\mathcal U,\mathcal U')\otimes \mathcal M_c(\mathcal U',\mathcal U'')\longrightarrow \mathcal M_c(\mathcal U,\mathcal U'')$ of $k$-coalgebras. From the proof of Theorem \ref{T2.7}, we also know that the morphism  $\mathcal M_c(\mathcal U,\mathcal U')\otimes \mathcal M_c(\mathcal U',\mathcal U'')\longrightarrow \mathcal M_c(\mathcal U,\mathcal U'')$ arises from the universal property of $\mathcal M_c(\mathcal U,\mathcal U'')$ applied to the measuring
$\mathcal M_c(\mathcal U,\mathcal U')\otimes \mathcal M_c(\mathcal U',\mathcal U'')\longrightarrow V(\mathcal U,\mathcal U')\otimes V(\mathcal U',\mathcal U'')\xrightarrow{\circ} V(\mathcal U,\mathcal U'')$. Hence, the  map $ \mathcal Q_{\mathcal M_c(\mathcal U,\mathcal U')\otimes \mathcal M_c(\mathcal U',\mathcal U'')}(P,P'')\longrightarrow
Vect_k(P,P'')$ gives a measuring when treated as a $\mathcal M_c(\mathcal U,\mathcal U'')$-comodule. The universal property of $\mathcal Q_{\mathcal M_c(\mathcal U,\mathcal U'')}(P,P'')$
as in Theorem \ref{T5.7}  now yields a morphism
\begin{equation}\label{ar532}
 (\mathcal M_c(\mathcal U,\mathcal U''), \mathcal Q_{\mathcal M_c(\mathcal U,\mathcal U')\otimes \mathcal M_c(\mathcal U',\mathcal U'')}(P,P''))\longrightarrow  (\mathcal M_c(\mathcal U,\mathcal U''), \mathcal Q_{\mathcal M_c(\mathcal U,\mathcal U'')}(P,P'') )
\end{equation} in $Comod_k$. Composing \eqref{ar532} with \eqref{ar531}, we obtain the required composition of Hom-objects $SAYD_k((\mathcal U,P),(\mathcal U',P'))\otimes SAYD_k((\mathcal U',P'),(\mathcal U'',P''))\longrightarrow SAYD_k((\mathcal U,P),(\mathcal U'',P''))$. This proves the result.
\end{proof}

\section{Comodule measurings and morphisms on cyclic (co)homology}

Throughout this section, we fix the following: let  $\mathcal U=(U,A_L,s_L,t_L,\Delta_L,\epsilon_L,S)$, 
and $\mathcal U'=(U',A_L',s'_L,t'_L,\Delta'_L,\epsilon'_L,S')$  be Hopf algebroids over $k$.  Let $P$ and $P'$ be SAYD modules over $\mathcal U$ and $\mathcal U'$ respectively. Let $(\Psi,\psi):C\longrightarrow V(\mathcal U,\mathcal U')$ be a cocommutative measuring and let $\Omega:D\longrightarrow Vect_k(P,P')$ be a $(C,\Psi,\psi)$-comodule measuring from $P$ to $P'$.

\smallskip
Since $\mathcal U$, $\mathcal U'$ are Hopf algebroids, we have recalled in Section 5 that the morphisms $\beta(\mathcal U):{_\blacktriangleright}U\otimes_{A^{op}}U_{\triangleleft}\longrightarrow U_\triangleleft \otimes_A {_\triangleright}U$ and $\beta(\mathcal U'):{_\blacktriangleright}U'\otimes_{A'^{op}}U'_{\triangleleft}\longrightarrow U'_\triangleleft \otimes_{A'} {_\triangleright}U'$  in the notation of \eqref{hg5.5} are bijections. We now need the following result.

\begin{lem}\label{L6.1}
For each $x\in C$, the following diagram commutes:
\begin{equation}\label{cd6.1}
\begin{CD}
U_\triangleleft \otimes_A {_\triangleright}U @>\beta(\mathcal U)^{-1}>>  {_\blacktriangleright}U\otimes_{A^{op}}U_{\triangleleft}\\
@VxVV @VxVV\\
 U'_\triangleleft \otimes_{A'} {_\triangleright}U'@>\beta(\mathcal U')^{-1}>>{_\blacktriangleright}U'\otimes_{A'^{op}}U'_{\triangleleft}\\
\end{CD}
\end{equation} Here, the left vertical map is given by $u^1\otimes_{A}u^2\mapsto x_{(1)}(u^1)\otimes_{A'}x_{(2)}(u^2)$ and the right vertical map
by $u^1\otimes_{A^{op}}u^2\mapsto x_{(1)}(u^1)\otimes_{A'^{op}}x_{(2)}(u^2)$.
\end{lem}
\begin{proof}
It is easy to verify that the vertical morphisms in \eqref{cd6.1} are well-defined. Further, since $\beta(\mathcal U)$ and $\beta(\mathcal U')$ are invertible, it suffices to check that the following diagram commutes
\begin{equation}\label{cd6.2e}
\begin{CD}
{_\blacktriangleright}U\otimes_{A^{op}}U_{\triangleleft}@>\beta(\mathcal U)>> U_\triangleleft \otimes_A {_\triangleright}U\\
@VxVV @VVxV\\
{_\blacktriangleright}U'\otimes_{A'^{op}}U'_{\triangleleft}@>\beta(\mathcal U')>> U'_\triangleleft \otimes_{A'} {_\triangleright}U'\\
\end{CD}
\end{equation} We now see that for $u\otimes_{A^{op}}v\in {_\blacktriangleright}U\otimes_{A^{op}}U_{\triangleleft}$ and $x\in C$, we have
\begin{equation*}
x(\beta(\mathcal U)(u\otimes_{A^{op}} v))=x(u_{(1)}\otimes_A u_{(2)}v)=x_{(1)}(u_{(1)})\otimes_{A'}x_{(2)}(u_{(2)})x_{(3)}(v)
=(x_{(1)}(u))_{(1)}\otimes_{A'}(x_{(1)}(u))_{(2)}x_{(2)}(v)=\beta(\mathcal U')(x_{(1)}(u)\otimes x_{(2)}(v))
\end{equation*} This proves the result.
\end{proof}

From Lemma \ref{L6.1}, it follows in the notation of \eqref{sty5.6} that we have 
\begin{equation}\label{pm6}
x_{(1)}(u_+)\otimes_{A'^{op}}x_{(2)}(u_-)=x(u_+\otimes_{A^{op}}u_-)=\beta(\mathcal U')^{-1}(x(u\otimes_A1))=x(u)_+\otimes_{A'^{op}}x(u)_-
\end{equation} for each $u\in U$. We now recall from \cite[Theorem 4.1]{KoKr} that the Hochschild homology groups
$HH_\bullet(\mathcal U;P)$ (resp. the cyclic homology groups $HC_\bullet(\mathcal U;P)$) of $\mathcal U$ with coefficients in the SAYD module $P$ are obtained from the cyclic module $C_\bullet(\mathcal U;P):=P\otimes_{A^{op}}({_\blacktriangleright}U_{\triangleleft})^{\otimes_{A^{op}}\bullet}$ with operators as follows
(where $\bar{u}:=u^1\otimes_{A^{op}}\otimes ...\otimes_{A^{op}}u^n$, $p\in P$)
\begin{equation}\label{facdeg6p}
\begin{array}{l}
d_i(p\otimes_{A^{op}}\bar{u}):=\left\{
\begin{array}{ll}
p\otimes_{A^{op}}u^1\otimes_{A^{op}}\dots \otimes_{A^{op}}u^{n-1}t_L(\epsilon_L(u^n)) &\qquad\mbox{if $i=0$}\\
p\otimes_{A^{op}}u^1\otimes_{A^{op}}\dots \otimes_{A^{op}}u^{n-i}u^{n-i+1}\otimes_{A^{op}}\dots &\qquad\mbox{if $1\leq i\leq n-1$}\\
pu^1\otimes_{A^{op}}u^2\otimes_{A^{op}}\dots \otimes_{A^{op}}u^n&\qquad\mbox{if $i=n$}\\
\end{array}\right.\\
s_i(p\otimes_{A^{op}}\bar{u}):=\left\{
\begin{array}{ll}
p\otimes_{A^{op}}u^1\otimes_{A^{op}}\dots \otimes_{A^{op}}u^n\otimes_{A^{op}}1&\qquad \mbox{if $i=0$}\\
p\otimes_{A^{op}} \dots \otimes_{A^{op}}u^{n-i}\otimes_{A^{op}}1\otimes_{A^{op}}u^{n-i+1}\otimes_{A^{op}}\dots &\qquad\mbox{if $1\leq i\leq n-1$}\\
p\otimes_{A^{op}}1\otimes_{A^{op}}u^1\otimes_{A^{op}}\dots \otimes_{A^{op}}u^n&\qquad\mbox{if $i=n$}\\
\end{array}\right.\\
t_n(p\otimes_{A^{op}}\bar{u}):=p_{(0)}u^1_+\otimes_{A^{op}}u^2_+\otimes_{A^{op}}\dots \otimes_{A^{op}}u^n_+\otimes_{A^{op}}u^n_-\dots u^1_-p_{(-1)}\\
\end{array}
\end{equation} We now have the following result.

\begin{thm}\label{P6.2h}
For each $y\in  D$, the family
\begin{equation}\label{6.5fq} \underline\Omega_n(y):C_n(\mathcal U;P)\longrightarrow C_n(\mathcal U';P') \qquad p\otimes u^1\otimes ...\otimes u^n\mapsto y(p\otimes u^1\otimes ...\otimes u^n)=y_{(0)}(p)\otimes y_{(1)}(u^1)\otimes ... \otimes y_{(n)}(u^n)
\end{equation} for $n\geq 0$ gives a morphism of cyclic modules. In particular, we have  induced morphisms
\begin{equation}\underline\Omega_\bullet^{hoc}(y):HH_\bullet(\mathcal U;P)\longrightarrow HH_\bullet(\mathcal U';P')\qquad \underline\Omega_\bullet^{cy}(y):HC_\bullet(\mathcal U;P)\longrightarrow HC_\bullet(\mathcal U';P')
\end{equation}
 on Hochschild and cyclic homologies for each $y\in  D$. 
\end{thm}

\begin{proof}
From the fact that $C$ is cocommutative and the conditions in Definition \ref{D5.6}, it is clear that the morphisms $\underline\Omega_n(y)$ are well defined, as well as the fact that they commute with the face maps and degeneracies appearing in the cyclic modules $C_\bullet(\mathcal U;P)$ and $C_\bullet(\mathcal U';P')$ 
as in \eqref{facdeg6p}. To verify that the morphisms in \eqref{6.5fq} also commute with the cyclic operators, we note that for
$p\otimes_{A^{op}}u^1\otimes_{A^{op}}\otimes ...\otimes_{A^{op}}u^n\in C_n(\mathcal U;P)$ we have
\begin{equation*}
\begin{array}{ll}
y(t_n(p\otimes u^1\otimes ...\otimes u^n))=y(p_{(0)}u^1_+\otimes u^2_+\otimes \dots \otimes u^n_+\otimes u^n_-\dots u^1_-p_{(-1)})&\\
=y_{(0)}(p_{(0)})y_{(1)}(u^1_+)\otimes y_{(2)}(u^2_+)\otimes \dots \otimes y_{(n)}(u^n_+)\otimes y_{(n+1)}(u^n_-)\dots y_{(2n)}(u^1_-)y_{(2n+1)}(p_{(-1)})&\\
=y_{(0)}(p_{(0)})y_{(2)}(u^1_+)\otimes y_{(4)}(u^2_+)\otimes \dots \otimes y_{(2n)}(u^n_+)\otimes y_{(2n+1)}(u^n_-)\dots y_{(3)}(u^1_-)y_{(1)}(p_{(-1)})& \mbox{(since $C$ is cocommutative)}\\
=y_{(0)}(p)_{(0)}y_{(1)}(u^1_+)\otimes y_{(3)}(u^2_+)\otimes \dots \otimes y_{(2n-1)}(u^n_+)\otimes y_{(2n)}(u^n_-)\dots y_{(2)}(u^1_-)y_{(0)}(p)_{(-1)}& \mbox{(using \eqref{cd523})}\\
=y_{(0)}(p)_{(0)}y_{(1)}(u^1)_+\otimes y_{(2)}(u^2)_+\otimes \dots \otimes y_{(n)}(u^n)_+\otimes y_{(n)}(u^n)_-\dots y_{(1)}(u^1)_-y_{(0)}(p)_{(-1)}& \mbox{(using \eqref{pm6})}\\
\end{array}
\end{equation*} This proves the result.
\end{proof}

We now come to cyclic cohomology. For this, we recall that  from \cite[Theorem 1.1, Theorem 3.6]{KoKr} that the Hochschild cohomology groups
$HH^\bullet(\mathcal U;P)$ (resp. the cyclic cohomology groups $HC^\bullet(\mathcal U;P)$) of $\mathcal U$ with coefficients in the SAYD module $P$ are obtained from the cocyclic module $C^\bullet(\mathcal U;P):=({_\triangleright}U_{\triangleleft})^{\otimes_{A}\bullet}\otimes_AP$ with operators as follows
(where $\bar{u}:=u^1\otimes_{A}\otimes ...\otimes_{A}u^n$, $p\in P$)
\begin{equation}\label{6.7yg}
\begin{array}{rl}
\delta_i(\bar{u}\otimes_Ap)&=\left\{
\begin{array}{ll}
1\otimes_Au^1\otimes_A\dots \otimes_Au^n\otimes_Ap &\qquad \mbox{if $i=0$}\\
u^1\otimes_A\dots \otimes_A\Delta_L(u^i)\otimes_A\dots \otimes_Au^n\otimes_Ap&\qquad \mbox{if $1\leq i\leq n$}\\
u^1\otimes_A\dots \otimes_Au^n\otimes_Ap_{(-1)}\otimes_Ap_{(0)}&\qquad \mbox{if $i=n+1$}\\
\end{array} \right.\\
\delta_i(p)&=\left\{
\begin{array}{ll}
1\otimes_Ap&\qquad \mbox{if $j=0$}\\
p_{(-1)}\otimes_Ap_{(0)}&\qquad \mbox{if $j=1$}\\
\end{array}\right.\\
\sigma_i(\bar{u}\otimes_Ap)&= u^1\otimes_A\dots \otimes_A\epsilon_L(u^{i+1})\otimes_A\dots \otimes_Au^n\otimes_Ap \qquad 0\leq i\leq n-1\\
\tau_n(\bar{u}\otimes_Ap) &= u^1_{-(1)}u^2\otimes_A \dots \otimes_Au^1_{-(n-1)}u^n\otimes_Au^1_{-(n)}p_{(-1)}\otimes_Ap_{(0)}u^1_+\\
\end{array}
\end{equation}

We now have the following result.

\begin{thm}\label{P6.3h}
For each $y\in  D$, the family
\begin{equation}\label{v6.5fq} \overline\Omega^n(y):C^n(\mathcal U;P)\longrightarrow C^n(\mathcal U';P') \qquad u^1\otimes ...\otimes u^n\otimes p\mapsto y(u^1\otimes ...\otimes u^n\otimes p)=y_{(1)}(u^1)\otimes ... \otimes y_{(n)}(u^n)\otimes y_{(0)}(p)
\end{equation} for $n\geq 0$ gives a morphism of cocyclic modules. In particular, we have  induced morphisms
\begin{equation}\overline\Omega^\bullet_{hoc}(y):HH^\bullet(\mathcal U;P)\longrightarrow HH^\bullet(\mathcal U';P')\qquad \overline\Omega^\bullet_{cy}(y):HC^\bullet(\mathcal U;P)\longrightarrow HC^\bullet(\mathcal U';P')
\end{equation}
 on Hochschild and cyclic cohomologies for each $y\in  D$. 
\end{thm}
\begin{proof} It is clear that the morphisms in \eqref{v6.5fq} are well-defined. 
For $y\in D$ and $i=n+1$ in \eqref{6.7yg}, we note that
\begin{equation}
\begin{array}{lll}
y(\delta_{n+1}(u^1\otimes\dots \otimes u^n\otimes p))&=y_{(1)}(u^1)\otimes \dots y_{(n)}(u^n)\otimes y_{(n+1)}(p_{(-1)})\otimes y_{(0)}(p_{(0)})&\\
&=y_{(2)}(u^1)\otimes \dots y_{(n+1)}(u^n)\otimes y_{(1)}(p_{(-1)})\otimes y_{(0)}(p_{(0)})& \mbox{(since $C$ is cocommutative)}\\
&=y_{(1)}(u^1)\otimes \dots y_{(n)}(u^n)\otimes (y_{(0)}(p))_{(-1)}\otimes y_{(0)}(p)_{(0)}&\mbox{(using \eqref{cd523})}\\
\end{array}
\end{equation} Similarly, we may verify that the morphisms in \eqref{v6.5fq} commute with the other coface and codegeneracy maps appearing in \eqref{6.7yg}. To show that they also commute with the cocyclic operators appearing in \eqref{6.7yg}, we note that for $u^1\otimes ...\otimes u^n\otimes p\in C^n(\mathcal U;P)$ and $
y\in D$, we have
\begin{equation*}
\begin{array}{ll}
y(\tau_n(u^1\otimes ...\otimes u^n\otimes p))=y( u^1_{-(1)}u^2\otimes_A \dots \otimes_Au^1_{-(n-1)}u^n\otimes_Au^1_{-(n)}p_{(-1)}\otimes_Ap_{(0)}u^1_+)&\\
= y_{(1)}(u^1_{-(1)})y_{(2)}(u^2)\otimes_A \dots \otimes_Ay_{(2n-3)}(u^1_{-(n-1)})y_{(2n-2)}(u^n)\otimes_Ay_{(2n-1)}(u^1_{-(n)})y_{(2n)}(p_{(-1)})\otimes_Ay_{(0)}(p_{(0)}u^1_+)&\\
= y_{(1)}(u^1_{-(1)})y_{(n+1)}(u^2)\otimes_A \dots \otimes_Ay_{(n-1)}(u^1_{-(n-1)})y_{(2n-1)}(u^n)\otimes_Ay_{(n)}(u^1_{-(n)})y_{(2n)}(p_{(-1)})\otimes_Ay_{(0)}(p_{(0)}u^1_+)&\\
= y_{(1)}(u^1_-)_{(1)}y_{(2)}(u^2)\otimes_A \dots \otimes_Ay_{(1)}(u^1_-)_{(n-1)}y_{(n)}(u^n)\otimes_Ay_{(1)}(u^1_-)_{(n)}y_{(n+1)}(p_{(-1)})\otimes_Ay_{(0)}(p_{(0)}u^1_+)&\\
= y_{(2)}(u^1_-)_{(1)}y_{(3)}(u^2)\otimes_A \dots \otimes_Ay_{(2)}(u^1_-)_{(n-1)}y_{(n+1)}(u^n)\otimes_Ay_{(2)}(u^1_-)_{(n)}y_{(n+2)}(p_{(-1)})\otimes_Ay_{(0)}(p_{(0)})y_{(1)}(u^1_+)&\\
= y_{(1)}(u^1)_{-(1)}y_{(2)}(u^2)\otimes_A \dots \otimes_Ay_{(1)}(u^1)_{-(n-1)}y_{(n)}(u^n)\otimes_Ay_{(1)}(u^1)_{-(n)}y_{(n+1)}(p_{(-1)})\otimes_Ay_{(0)}(p_{(0)})y_{(1)}(u^1)_+&\mbox{(using \eqref{pm6})}\\
= y_{(1)}(u^1)_{-(1)}y_{(2)}(u^2)\otimes_A \dots \otimes_Ay_{(1)}(u^1)_{-(n-1)}y_{(n)}(u^n)\otimes_Ay_{(1)}(u^1)_{-(n)}y_{(0)}(p)_{(-1)}\otimes_Ay_{(0)}(p)_{(0)}y_{(1)}(u^1)_+&\mbox{(using \eqref{cd523})}\\
\end{array}
\end{equation*} This proves the result.
\end{proof}

Finally, we recall from \cite[$\S$ 4.3]{KoKr} that there are Hopf-Galois isomorphisms relating the modules $C_\bullet(\mathcal U;P)$ and $C^\bullet(\mathcal
U;P)$
\begin{equation}\label{hgal6t}
\xi_n(\mathcal U;P):C_n(\mathcal U;P)\overset{\cong}{\longrightarrow}C^n(\mathcal U;P)\qquad p\otimes u^1\otimes \dots \otimes u^n\mapsto u^1_{(1)}
\otimes u^1_{(2)}u^2_{(1)}\otimes \dots \otimes u^1_{(n)}u^2_{(n-1)}\dots u^{n-1}_{(2)}u^n\otimes p
\end{equation} We will conclude this section by showing that the morphisms induced by comodule measurings of SAYD modules are compatible with the Hopf-Galois isomorphisms in
\eqref{hgal6t}.

\begin{Thm}\label{T6.3} Let  $\mathcal U=(U,A_L,s_L,t_L,\Delta_L,\epsilon_L,S)$, 
and $\mathcal U'=(U',A_L',s'_L,t'_L,\Delta'_L,\epsilon'_L,S')$  be Hopf algebroids over $k$.  Let $P$ and $P'$ be SAYD modules over $\mathcal U$ and $\mathcal U'$ respectively. Let $(\Psi,\psi):C\longrightarrow V(\mathcal U,\mathcal U')$ be a cocommutative measuring and let $\Omega:D\longrightarrow Vect_k(P,P')$ be a $(C,\Psi,\psi)$-comodule measuring from $P$ to $P'$. Then, for each $y\in D$, the following diagram commutes
\begin{equation}
\begin{CD}
C_n(\mathcal U;P) @>\xi_n(\mathcal U;P)>> C^n(\mathcal U;P)\\
@V\underline{\Omega}_n(y)VV @VV\overline{\Omega}^n(y)V\\
C_n(\mathcal U';P') @>\xi_n(\mathcal U';P')>> C^n(\mathcal U';P')\\
\end{CD}
\end{equation}

\end{Thm}

\begin{proof} We set $N:=n(n-1)/2$. For $y\in D$ and $p\otimes u^1\otimes \dots \otimes u^n\in C_n(\mathcal U;P)$, we see that
\begin{equation}
\begin{array}{ll}
\overline{\Omega}^n(y)(\xi_n(\mathcal U;P)(p\otimes u^1\otimes \dots \otimes u^n)) &\\
=\overline{\Omega}^n(y)(u^1_{(1)}
\otimes u^1_{(2)}u^2_{(1)}\otimes \dots \otimes u^1_{(n)}u^2_{(n-1)}\dots u^{n-1}_{(2)}u^n\otimes p)&\\
=y_{(1)}(u^1_{(1)})
\otimes y_{(2)}(u^1_{(2)})y_{(3)}(u^2_{(1)})\otimes \dots \otimes y_{(N+1)}(u^1_{(n)})y_{(N+2)}(u^2_{(n-1)})\dots y_{(N+n-1)}(u^{n-1}_{(2)})y_{(N+n)}(u^n)\otimes y_{(0)}(p)&\\
=y_{(1)}(u^1_{(1)})
\otimes y_{(2)}(u^1_{(2)})y_{(n+1)}(u^2_{(1)})\otimes \dots \otimes y_{(n)}(u^1_{(n)})y_{(2n-1)}(u^2_{(n-1)})\dots y_{(N+n-1)}(u^{n-1}_{(2)})y_{(N+n)}(u^n)\otimes y_{(0)}(p)&\\
=y_{(1)}(u^1)_{(1)}
\otimes y_{(1)}(u^1)_{(2)}y_{(2)}(u^2)_{(1)}\otimes \dots \otimes y_{(1)}(u^1)_{(n)}y_{(2)}(u^2)_{(n-1)}\dots y_{(n-1)}(u^{n-1})_{(2)}y_{(n)}(u^n)\otimes y_{(0)}(p)&\\
=\xi_n(\mathcal U;P)(\underline{\Omega}_n(y)(p\otimes u^1\otimes \dots \otimes u^n))&\\
\end{array}
\end{equation}

\end{proof}

\section{Measurings of Lie-Rinehart algebras and morphisms in   homology}

Throughout this section, we assume that the ground field $k$ contains $\mathbb Q$. By definition, a Lie-Rinehart algebra over $k$ (see, for instance, \cite{H1}, \cite{H2}, \cite{Ri})  consists of a pair  $(R,\mathfrak L)$ such that

\smallskip
(a) $R$ is a commutative $k$-algebra and $\mathfrak L$ is a $R$-module

\smallskip
(b) $\mathfrak L$ carries the structure of a  Lie algebra over $k$ and is equipped with a morphism $\mathfrak L\longrightarrow Der_k(R)$, $Z\mapsto \{r\mapsto Z(r)\}$ of Lie algebras satisfying
\begin{equation}
(rZ)(r')=r(Z(r'))\qquad [Z,rZ']=r[Z,Z']+Z(r)Z'\qquad r,r'\in R,\textrm{ }Z,Z'\in \mathfrak L
\end{equation} A (right) connection on $(R,\mathfrak L)$ consists of a linear map $\nabla:\mathfrak L\otimes R \longrightarrow R$ that satisfies (see \cite{H2}, \cite{KoP})
\begin{equation}\label{72di}
\nabla(Z\otimes r'r)=r'\nabla(Z\otimes r)-Z(r')r=\nabla(r'Z\otimes r)\qquad Z\in \mathfrak L, \textrm{ }r,r'\in R
\end{equation} Since $\nabla$ may also be written as a map $\nabla:\mathfrak L\longrightarrow Vect_k(R,R)$, $Z\mapsto \nabla_Z$, the relation in \eqref{72di} may also be expressed as
$\nabla_Z(r'r)=r'\nabla_Z(r)-Z(r')r=\nabla_{r'Z}(r)$. The connection is said to be flat if $[\nabla_Z,\nabla_{Z'}]=\nabla_{[Z',Z]}$ for any $Z$, $Z'\in \mathfrak L$. In this paper, we will only consider Lie-Rinehart algebras $(R,\mathfrak L)$ equipped with a flat connection $\nabla$ and such that $\mathfrak L$ is projective as a 
$R$-module. For more on Lie-Rinehart algebras, we refer the reader, for instance, to \cite{H1}, \cite{H2}, \cite{KoKr2}, \cite{MM}, \cite{Xu0}, \cite{Xu}.

\begin{defn}\label{llD7.1}
Let $(R,\mathfrak L,\nabla)$ and $(R',\mathfrak L',\nabla')$ be Lie-Rinehart algebras over $k$. Let $C$ be a cocommutative 
$k$-coalgebra. Then, a  $C$-measuring $(\Psi,\psi)$ from  $(R,\mathfrak L,\nabla)$ to $(R',\mathfrak L',\nabla')$ consists of a pair of linear maps
\begin{equation}
\begin{array}{c}
\Psi: C\longrightarrow Vect_k(\mathfrak L,\mathfrak L')\qquad x\mapsto \Psi(x)=\{Z\mapsto x(Z)\} \\
\psi:C\longrightarrow Vect_k(R,R') \qquad x\mapsto \psi(x)=\{r\mapsto x(r)\}\\
\end{array}
\end{equation} for $x\in C$, $r\in R$ and $Z\in \mathfrak L$, such that 

\smallskip
(a) $\psi:C\longrightarrow Vect_k(R,R')$ is a measuring of algebras

\smallskip
(b) For $Z_1$, $Z_2\in \mathfrak L$ and $x\in C$, we have
$
x([Z_1,Z_2])=[x_{(1)}(Z_1),x_{(2)}(Z_2)]
$.

\smallskip
(c) For $x\in C$, $r\in R$ and $Z\in \mathfrak L$, we have 
\begin{equation}\label{74ds}
x(Z(r))=(x_{(1)}(Z))(x_{(2)}(r))\qquad x(\nabla_Z(r))=\nabla_{x_{(1)}(Z)}(x_{(2)}(r))\qquad x(rZ)=x_{(1)}(r)x_{(2)}(Z)
\end{equation}
\end{defn}

Let $(R,\mathfrak L,\nabla)$ be a Lie-Rinehart algebra. Then, its homology is computed by taking the exterior algebra $\wedge^\bullet_R\mathfrak L$ along with the differential
$\partial: \wedge^n_R\mathfrak L\longrightarrow \wedge^{n-1}_R\mathfrak L$ given by setting (see \cite[$\S$ 3.3.3]{KoP})
\begin{equation}\label{lrhom}
\begin{array}{ll}
\partial(Z_1\wedge ... \wedge Z_n)&=\sum_{i=1}^n(-1)^{i+1}\nabla_{Z_i}(1_R)Z_1\wedge ... \wedge \hat{Z}_i\wedge ...\wedge Z_n\\
&\quad + \sum_{i<j} (-1)^{i+j}[Z_i,Z_j]\wedge Z_1\wedge ...\wedge \hat{Z}_i\wedge...\wedge \hat{Z}_j\wedge ...\wedge Z_n\\
\end{array}
\end{equation} for   $Z_1,...,Z_n\in \mathfrak L$. The homology groups of this complex will be denoted by $H_\bullet(R,\mathfrak L,\nabla)$. 

\begin{thm}\label{skP7.2}
Let $C$ be a cocommutative coalgebra and let $(\Psi,\psi)$ be a $C$-measuring of Lie-Rinehart algebras from  $(R,\mathfrak L,\nabla)$ to $(R',\mathfrak L',\nabla')$. Then, for each $x\in C$, the family of maps (for $n\geq 0$)
\begin{equation}\label{76vu}
\wedge^n_R\mathfrak L\longrightarrow \wedge^n_{R'}\mathfrak L' \qquad Z_1\wedge ... \wedge Z_n\mapsto x_{(1)}(Z_1)\wedge ... \wedge x_{(n)}(Z_n)
\end{equation} gives a morphism of complexes from $(\wedge^\bullet_R\mathfrak L,\partial)$ to $(\wedge^\bullet_{R'}\mathfrak L',\partial')$. Accordingly, for each $x\in C$, we have  induced morphisms $\underline{\Psi}_\bullet^{Lie}(x):H_\bullet(R,\mathfrak L,\nabla)
\longrightarrow H_\bullet(R',\mathfrak L',\nabla')$ on homology groups.
\end{thm}

\begin{proof}
Since $C$ is cocommutative, it is evident that the maps in \eqref{76vu} are well-defined for each $x\in C$. We now verify that
\begin{equation*}
\begin{array}{l}
x(\partial(Z_1\wedge ... \wedge Z_n))\\
=\sum_{i=1}^n(-1)^{i+1}x_{(1)}(\nabla_{Z_i}(1_R)Z_1)\wedge ... \wedge x_{(i)}(\hat{Z}_i)\wedge ...\wedge x_{(n-1)}(Z_n)\\
\textrm{ }+ \sum_{i<j} (-1)^{i+j}x_{(1)}([Z_i,Z_j])\wedge x_{(2)}(Z_1)\wedge ...\wedge x_{(i+1)}(\hat{Z}_i)\wedge...\wedge x_{(j)}(\hat{Z}_j)\wedge ...\wedge x_{(n-1)}(Z_n)\\
=\sum_{i=1}^n(-1)^{i+1}x_{(1)}(\nabla_{Z_i}(1_R))x_{(2)}(Z_1)\wedge ... \wedge x_{(i+1)}(\hat{Z}_i)\wedge ...\wedge x_{(n)}(Z_n) \\
 \textrm{ }+ \sum_{i<j} (-1)^{i+j}x_{(1)}([Z_i,Z_j])\wedge x_{(2)}(Z_1)\wedge ...\wedge x_{(i+1)}(\hat{Z}_i)\wedge...\wedge x_{(j)}(\hat{Z}_j)\wedge ...\wedge x_{(n-1)}(Z_n)\\
=\sum_{i=1}^n(-1)^{i+1} (\nabla_{x_{(1)}(Z_i)}(1_{R'}))x_{(2)}(Z_1)\wedge ... \wedge x_{(i+1)}(\hat{Z}_i)\wedge ...\wedge x_{(n)}(Z_n)\\ \textrm{ } + \sum_{i<j} (-1)^{i+j}[x_{(1)}(Z_i),x_{(2)}(Z_j)]\wedge x_{(3)}(Z_1)\wedge ...\wedge x_{(i+2)}(\hat{Z}_i)\wedge...\wedge x_{(j+1)}(\hat{Z}_j)\wedge ...\wedge x_{(n)}(Z_n)\\
=\sum_{i=1}^n(-1)^{i+1} (\nabla_{x_{(i)}(Z_i)}(1_{R'}))x_{(1)}(Z_1)\wedge ... \wedge x_{(i)}(\hat{Z}_i)\wedge ...\wedge x_{(n)}(Z_n)\\ \textrm{ } + \sum_{i<j} (-1)^{i+j}[x_{(i)}(Z_i),x_{(j)}(Z_j)]\wedge x_{(1)}(Z_1)\wedge ...\wedge x_{(i)}(\hat{Z}_i)\wedge...\wedge x_{(j)}(\hat{Z}_j)\wedge ...\wedge x_{(n)}(Z_n)\\
\end{array}
\end{equation*} This proves the result.
\end{proof}

If $(R,\mathfrak L,\nabla)$ is a Lie-Rinehart algebra, its universal enveloping algebra $\mathcal V(R,\mathfrak L,\nabla)$ is constructed in the following steps (see, for instance, 
\cite{KoP}, \cite{MM}, \cite{Xu}): the direct sum $R\oplus \mathfrak L$ is made into a $k$-Lie-algebra by setting
\begin{equation}\label{pluslie}
[(r_1,Z_1),(r_2,Z_2)]:=((Z_1(r_2)-Z_2(r_1),[Z_1,Z_2])
\end{equation} Accordingly, one can form the universal enveloping algebra $\mathcal U(R\oplus\mathfrak L)$ and let $\bar{\mathcal U}(R\oplus \mathfrak L)$ be its subalgebra generated by elements in the image of $R\oplus\mathfrak L$ in $\mathcal U(R\oplus\mathfrak L)$. For any
$(r,Z)\in R\oplus \mathfrak L$, let $\overline{(r,Z)}$ be its image in $\bar{\mathcal U}(R\oplus \mathfrak L)$. Then, $\mathcal V(R,\mathfrak L,\nabla)$  is formed by taking the quotient of $\bar{\mathcal U}(R\oplus \mathfrak L)$ over the two sided ideal generated by elements of the form
\begin{equation}\label{twos7}
\overline{(r_1r_2,r_1Z_2)}-\overline{r_1}\otimes \overline{(r_2,Z_2)}
\end{equation} for $r_1\in R$, $(r_2,Z_2)\in R\oplus \mathfrak L$. The $\otimes$ symbol is used in \eqref{twos7} because the product in $\bar{\mathcal U}(R\oplus \mathfrak L)$ is induced from the product in the tensor algebra over the space $R\oplus \mathfrak L$. 

\begin{lem}\label{L7.3px}
Let $C$ be a cocommutative coalgebra and let $(\Psi,\psi)$ be a $C$-measuring of Lie-Rinehart algebras from  $(R,\mathfrak L,\nabla)$ to $(R',\mathfrak L',\nabla')$. Then, $(\Psi,\psi)$ induces a measuring $\mathcal V(\Psi,\psi):C\longrightarrow Vect_k(\mathcal V(R,\mathfrak L,\nabla),\mathcal V(R',\mathfrak L',\nabla'))$ of universal enveloping algebras.
\end{lem}

\begin{proof}
For each $x\in C$, we use the measuring $(\Psi,\psi)$ to define a map $R\oplus \mathfrak L\longrightarrow R'\oplus\mathfrak L'$, $(r,Z)\mapsto x(r,Z):=(\psi(x)(r),\Psi(x)(Z))=
(x(r),x(Z))$. This means that
\begin{equation}\label{79ph}
\begin{array}{ll}
x([(r_1,Z_1),(r_2,Z_2)])=x((Z_1(r_2)-Z_2(r_1),[Z_1,Z_2]) &= (x(Z_1(r_2))-x(Z_2(r_1)),x([Z_1,Z_2]))\\
&=(x_{(1)}(Z_1)(x_{(2)}(r_2))-x_{(1)}(Z_2)(x_{(2)}(r_1)),[x_{(1)}(Z_1),x_{(2)}(Z_2)]) \\
&=(x_{(1)}(Z_1)(x_{(2)}(r_2))-x_{(2)}(Z_2)(x_{(1)}(r_1)),[x_{(1)}(Z_1),x_{(2)}(Z_2)]) \\
&=[x_{(1)}(r_1,Z_1),x_{(2)}(r_2,Z_2)]\\
\end{array}
\end{equation} From \eqref{79ph}, it follows that there is a measuring of algebras from $\mathcal U(R\oplus\mathfrak L)$ to $\mathcal U(R'\oplus\mathfrak L')$ which is induced by
(for $x\in C$)
\begin{equation}\label{710ten}
x((r_1,Z_1)\otimes ... \otimes (r_n,Z_n)):=(x_{(1)}(r_1,Z_1)\otimes ... \otimes x_{(n)}(r_n,Z_n))
\end{equation}
at the level of the tensor algebra over $R\oplus \mathfrak L$. Since this measuring carries elements of $R\oplus \mathfrak L$ to $R'\oplus \mathfrak L'$, we can restrict it to a measuring of algebras from $\bar{\mathcal U}(R\oplus \mathfrak L)$ to $\bar{\mathcal U}(R'\oplus \mathfrak L')$. We now consider an element $\overline{(r_1r_2,r_1Z_2)}-\overline{r_1}\otimes \overline{(r_2,Z_2)}$ in $\bar{\mathcal U}(R\oplus \mathfrak L)$ of the form \eqref{twos7}. For any $x\in C$, we have
\begin{equation*}
\begin{array}{ll}
x(\overline{(r_1r_2,r_1Z_2)}-\overline{r_1}\otimes \overline{(r_2,Z_2)})&=\overline{(x(r_1r_2),x(r_1Z_2))}-\overline{x_{(1)}(r_1)}\otimes \overline{x_{(2)}(r_2,Z_2)}\\
&=\overline{(x_{(1)}(r_1)x_{(2)}(r_2),x_{(1)}(r_1)x_{(2)}(Z_2))}-\overline{x_{(1)}(r_1)}\otimes \overline{(x_{(2)}(r_2),x_{(2)}(Z_2))}=0\in \mathcal V(R',\mathfrak L',\nabla')\\
\end{array}
\end{equation*} It is now clear that the measuring from $\bar{\mathcal U}(R\oplus \mathfrak L)$ to $\bar{\mathcal U}(R'\oplus \mathfrak L')$ descends to their respective quotients, which gives a measuring of algebras from $\mathcal V(R,\mathfrak L,\nabla)$ to $\mathcal V(R',\mathfrak L',\nabla')$. 
\end{proof}

By the definition, it is clear that there exists a canonical inclusion $R\hookrightarrow \mathcal V(R,\mathfrak L,\nabla)$.  We know from \cite[$\S$ 3.3.2]{KoP}, \cite[$\S$ 5.1]{KoKr} that the universal enveloping algebra $\mathcal V(R,\mathfrak L,\nabla)$ becomes a left Hopf algebroid $(\mathcal V(R,\mathfrak L,\nabla),R,s_L,t_L,\Delta_L,\epsilon_L)$ determined by
\begin{align}\label{bil7a}
s_L=t_L:R\hookrightarrow  \mathcal V(R,\mathfrak L,\nabla)\qquad r\mapsto r\\
\label{bil7ep} \epsilon_L(Z)= 0\qquad \epsilon_L(r)=r\qquad r\in R, Z\in  \mathfrak L \\
\label{bil7b}\Delta_L(Z)=1\otimes Z+Z\otimes 1 \qquad \Delta_L(r)=r\otimes 1\qquad  r\in R,Z\in \mathfrak L
\end{align} We refer to \cite{KoP} for the complete notion of a  left Hopf algebroid, which is a left  bialgebroid satisfying certain conditions. Accordingly, the notion of coalgebra measuring of left Hopf algebroids, or that of comodule measurings of SAYD modules over left Hopf algebroids, may be obtained by modifying Definitions \ref{D2.3} and \ref{D5.6} in an obvious manner. We observe that $t_L$ in \eqref{bil7a} is also an anti-homomorphism because $R$ is commutative. 

\smallskip
Given $(R,\mathfrak L,\nabla)$, we know (see \cite[$\S$ 3.3.2]{KoP}) that $R$ becomes a right $\mathcal V(R,\mathfrak L,\nabla)$-module with action determined by the usual multiplication on $R$ as well
as $r\cdot Z:=\nabla_Z(r)$ for $r\in R$, $Z\in \mathfrak L$. In fact, together with the left comodule action $\Delta_R:R\longrightarrow \mathcal V(R,\mathfrak L,\nabla)\otimes_R R$, $r\mapsto 1\otimes r$, we know that $R$ becomes an SAYD-module over $\mathcal V(R,\mathfrak L,\nabla)$ (see \cite[Lemma 5.1]{KoKr}). 

\begin{lem}\label{L7.5ff}
Let $C$ be a cocommutative coalgebra and let $(\Psi,\psi)$ be a $C$-measuring of Lie-Rinehart algebras from  $(R,\mathfrak L,\nabla)$ to $(R',\mathfrak L',\nabla')$. Then, the triple:
\begin{equation}\label{msd7k}
\psi:C\longrightarrow Vect_k(R,R')\qquad \mathcal V(\Psi,\psi):C\longrightarrow Vect_k(\mathcal V(R,\mathfrak L,\nabla),\mathcal V(R',\mathfrak L',\nabla')) \qquad \psi:C\longrightarrow Vect_k(R,R')
\end{equation} is the data of a comodule measuring from $R$ to $R'$. 
\end{lem}

\begin{proof}
By  Lemma \ref{L7.3px} and the definitions in \eqref{bil7a}-\eqref{bil7b},  it is clear that $(\mathcal V(\Psi,\psi),\psi)$ is a measuring of left Hopf algebroids. Using \eqref{74ds} and the fact that $C$ is cocommutative, we have for any $x\in C$, $r\in R$ and $Z\in
\mathfrak L$:
\begin{equation}\label{718bl}
x(r\cdot Z)=x(\nabla_Z(r))=\nabla_{x_{(1)}(Z)}(x_{(2)}(r))=\nabla_{x_{(2)}(Z)}(x_{(1)}(r))=(x_{(1)}(r))\cdot x_{(2)}(Z)
\end{equation} It follows from \eqref{718bl} that $(\mathcal V(\Psi,\psi),\psi)$ is a comodule measuring from the right $\mathcal V(R,\mathfrak L,\nabla)$ module $R$ to the
right $\mathcal V(R',\mathfrak L',\nabla')$-module $R'$. Finally, for $x\in C$ and $r\in R$, we have
$
\Delta_{R'}(x(r))=1\otimes x(r)=x(1\otimes r)=x(\Delta_R(r))
$ and hence the condition in \eqref{cd523} is satisfied. This proves the result.
\end{proof}

Let $C^\bullet(\mathcal V(R,\mathfrak L,\nabla);R)$ be the cocyclic module corresponding to the left  Hopf algebroid $\mathcal V(R,\mathfrak L,\nabla)$ and the SAYD module $R$ as in \eqref{6.7yg}. Then, we know from \cite[Theorem 5.2]{KoKr} that the maps
\begin{equation}\label{ant7}
Alt_n: \wedge^n_R\mathfrak L\longrightarrow \mathcal V(R,\mathfrak L,\nabla)^{\otimes_Rn}\qquad Z_1\wedge ...\wedge Z_n\mapsto \frac{1}{n!}\underset{\sigma\in S_n}{\sum}(-1)^\sigma Z_{\sigma(1)}\otimes ...\otimes Z_{\sigma(n)}
\end{equation} induce a quasi-isomorphism of mixed complexes $(\wedge_R^\bullet\mathfrak L,0,\partial)\longrightarrow (C^\bullet(\mathcal V(R,\mathfrak L,\nabla);R),b,B)$. Here, $b$ is the standard Hochschild differential and $B$ is Connes' differential (see, for instance, \cite[$\S$ 2.5.13]{Loday}).

\begin{Thm}\label{T7fin}
Let $C$ be a cocommutative coalgebra and let $(\Psi,\psi)$ be a $C$-measuring of Lie-Rinehart algebras from  $(R,\mathfrak L,\nabla)$ to $(R',\mathfrak L',\nabla')$. Then, for each $x\in C$, the following diagram is commutative
\begin{equation}\label{cd717h}
\begin{CD}
\underset{\mbox{$n\equiv \bullet$ (mod $2$) }}{\bigoplus}H_n(R,\mathfrak L,\nabla) @>\cong>> HP^\bullet(\mathcal V(R,\mathfrak L,\nabla);R)\\
@V\underset{\mbox{\tiny $n\equiv \bullet$ (mod $2$) }}{\oplus}\underline{\Psi}_n^{Lie}(x)VV @VV\overline{\Psi}^\bullet_{per}(x)V\\
\underset{\mbox{$n\equiv \bullet$ (mod $2$) }}{\bigoplus}H_n(R',\mathfrak L',\nabla') @>\cong>> HP^\bullet(\mathcal V(R',\mathfrak L',\nabla);R')\\
\end{CD}
\end{equation} Here the right vertical arrow $\overline{\Psi}^\bullet_{per}(x)$ is the induced morphism on periodic cyclic cohomology induced by the morphism of cocyclic modules in
Proposition \ref{P6.3h}. The horizontal arrows are the induced by the quasi-isomorphism in \eqref{ant7}. 
\end{Thm}

\begin{proof}
The horizontal isomorphisms in \eqref{cd717h} follow from the quasi-isomorphism in \eqref{ant7} and \cite[Theorem 3.13]{KoP}, \cite[Theorem 5.2]{KoKr}. We need to verify that the maps in \eqref{ant7} commute with the action of the measuring $(\Psi,\psi)$. Since $C$ is cocommutative, we note that for $x\in C$ and $Z_1\wedge ...\wedge Z_n
\in  \wedge^n_R\mathfrak L$, we have
\begin{equation*}
\begin{array}{ll}
Alt_n(x(Z_1\wedge ...\wedge Z_n))&=Alt_n(x_{(1)}(Z_1)\wedge ... \wedge x_{(n)}(Z_n))\\
&=(1/n!)\underset{\sigma\in S_n}{\sum}(-1)^\sigma x_{\sigma(1)}(Z_{\sigma(1)})\otimes ...\otimes x_{\sigma(n)}(Z_{\sigma(n)} )=(1/n!)\underset{\sigma\in S_n}{\sum}(-1)^\sigma x_{(1)}(Z_{\sigma(1)})\otimes ...\otimes x_{(n)}(Z_{\sigma(n)} )\\
&=x(Alt_n(Z_1\wedge ...\wedge Z_n))
\end{array}
\end{equation*} This proves the result.
\end{proof}

\section{Operads with multiplication, comp modules and morphisms on cyclic homology}

We start this final section by recalling the notion of a non-symmetric operad (or non-$\Sigma$ operad) with multiplication. In this section, we will only consider such operads. For more on this subject, we refer the reader, for instance, to \cite{Fres}, \cite{GerS}, \cite{Lodv}, \cite{Men}. 

\begin{defn}\label{D7.1} (see, for instance, \cite[Definition 2.2]{Ko4})  A non-$\Sigma$ operad $(\mathscr  O,m,e)$ over $k$ with multiplication consists of the following:

\smallskip
(a) A collection of vector spaces $\mathscr O=\{\mathscr O(n)\}_{n\geq 0}$. 

\smallskip
(b) A family of $k$-bilinear operations $\circ_i:\mathscr O(p)\otimes \mathscr O(q)\longrightarrow \mathscr O(p+q-1)$ and an identity $\mathbbm{1}\in \mathscr O(1)$  satisfying the following conditions
(for $u\in \mathscr O(p)$, $v\in \mathscr O(q)$, $w\in \mathscr O(r)$)
\begin{equation}
\begin{array}{lll}
u\circ_iv & = 0 & \mbox{if $p<i$ or $p=0$}\\
(u\circ_iv)\circ_jw&=\left\{\begin{array}{ll}
(u\circ_jw)\circ_{i+r-1}v& \mbox{if $j<i$}\\
u\circ_i(v\circ_{j-i+1}w) & \mbox{if $i\leq j<q+i$}\\
(u\circ_{j-q+1}w)\circ_iv & \mbox{if $j\geq q+i$}\\
\end{array} \right. &\\
u\circ_i\mathbbm{1}&=\mathbbm{1}\circ_1u=u&\mbox{for $i\leq p$}\\
\end{array}
\end{equation}

\smallskip
(c) An operad multiplication $m\in \mathscr O(2)$ and a unit $e\in \mathscr O(0)$ such that
\begin{equation}
m\circ_1m=m\circ_2m \qquad m\circ_1e=m\circ_2e=\mathbbm{1}
\end{equation}

\end{defn}

We now consider the notion of a cocommutative measuring between non-symmetric operads with multiplication.

\begin{defn}\label{D7.11h}
Let $C$ be a cocommutative $k$-coalgebra and let $(\mathscr O,m,e)$, $(\mathscr O',m',e')$ be non-$\Sigma$ operads with multiplication. A $C$-measuring from $(\mathscr O,m,e)$ to $(\mathscr O',m',e')$ consists of a family of maps $\{\Psi_n:C\longrightarrow Vect_k(\mathscr O(n),\mathscr O'(n)), \textrm{ }x\mapsto x_n:=\Psi_n(x):\mathscr O(n)\longrightarrow \mathscr O'(n)\}_{n\geq 0}$ such that 
\begin{equation}\label{eq7.2ctn}
\begin{array}{c}
x_{p+q-1}(u\circ_iv):=\Psi_{p+q-1}(x)(u\circ_iv)=\Psi_{p}(x_{(1)})(u)\circ'_i\Psi_q(x_{(2)})(v)=\sum x_{(1),p}(u)\circ'_i x_{(2),q}(v)\\ x_2(m)=\Psi_2(x)(m)=\epsilon_C(x) m'\qquad x_0(e)=\Psi_0(x)(e)=\epsilon_C(x)e'
\end{array}
\end{equation} for $u\in \mathscr O(p)$, $v\in \mathscr O(q)$ and $x\in C$, where $\Delta_C(x)=x_{(1)}\otimes x_{(2)}$ is the coproduct on $C$ and $\epsilon_C:C\longrightarrow k$ is the counit on $C$.
\end{defn}

\begin{thm}\label{P7.3dq}
Let $(\mathscr O,m,e)$, $(\mathscr O',m',e')$ be non-$\Sigma$ operads with multiplication. Then, 

\smallskip
(a) There exists a cocommutative coalgebra $\mathcal M_c(\mathscr O,\mathscr O')$ and a measuring
$\{\Phi_n:\mathcal M_c(\mathscr O,\mathscr O')\longrightarrow Vect_k(\mathscr O(n),\mathscr O'(n))\}_{n\geq 0}$ satisfying the following universal property: given any measuring $\{\Psi_n:C\longrightarrow Vect_k(\mathscr O(n),\mathscr O'(n))\}_{n\geq 0}$ with a cocommutative coalgebra $C$, there exists a unique morphism $\xi:C\longrightarrow \mathcal M_c(\mathscr O,\mathscr O')$ of coalgebras making the following diagram commutative for each $n\geq 0$
\begin{equation}
\xymatrix{
\mathcal M_c(\mathscr O,\mathscr O')\ar[rr]^{\Phi_n} && Vect_k(\mathscr O(n),\mathscr O'(n))\\
&C\ar[ul]^\xi\ar[ur]_{\Psi_n}&\\
}
\end{equation}

\smallskip
(b) The  non-$\Sigma$ operads over $k$ with multiplication form a category that is enriched over the category $CoCoalg_k$ of cocommutative $k$-coalgebras. 
\end{thm} 

\begin{proof}
(a) We set $V:=\underset{n\geq 0}{\prod} Vect_k( \mathscr O(n),\mathscr O'(n))$. Let $\pi(V):\mathfrak C(V)\longrightarrow V$ be the cofree coalgebra over $V$. We now set $\mathcal M_c(\mathscr O,\mathscr O'):=\sum D$, where the sum is taken over all cocommutative subcoalgebras of $\mathfrak C(V)$ such that the restriction 
of $\pi(V)$ to $D$ gives a measuring from $\mathscr O$ to $\mathscr O'$. The universal property of $\mathcal M_c(\mathscr O,\mathscr O')$ now follows as in the proof of
Proposition \ref{P2.5}. 

\smallskip
(b) For $(\mathscr O,m,e)$, $(\mathscr O',m',e')$, the hom-object in $CoCoalg_k$ is given by $\mathcal M_c(\mathscr O,\mathscr O')$. Further, the scalar multiples of the identity give a measuring from $\mathscr O$ to itself, which induces a morphism $k\longrightarrow \mathcal M_c(\mathscr O,\mathscr O)$ of coalgebras by part (a). The result now follows as in the proof of Theorem \ref{T2.7}. 
\end{proof}

\begin{defn}\label{D7.2} (see \cite[Definition 3.1]{Ko4}) A cyclic unital (left) comp module $\mathscr L$ over an operad $(\mathscr O,m,e)$ with multiplication consists of the following data:

\smallskip
(a) A collection of vector spaces $\mathscr L=\{\mathscr L(n)\}_{n\geq 0}$. 

\smallskip
(b) A family of $k$-bilinear operations 
\begin{equation} \bullet_i:\mathscr O(p)\otimes \mathscr L(n)\longrightarrow \mathscr (n-p+1), \textrm{ } 0\leq i\leq n+1-p
\end{equation} set to be zero for $p\geq n+1$ and satisfying the following conditions for
$u\in \mathscr O(p)$, $v\in \mathscr O(q)$, $l\in \mathscr L(n)$
\begin{equation*}
\begin{array}{ll}
u\bullet_i(v\bullet_jl)=\left\{\begin{array}{ll}
v\bullet_j(u\bullet_{i+q-1}l)&\mbox{$j<i$}\\
(u\circ_{j-i+1}v)\bullet_il& \mbox{if $j-p<i\leq j$}\\
v\bullet_{j-p+1}(u\bullet_il)&\mbox{if $0\leq i\leq j-p$}\\
\end{array}\right.  & \qquad\mbox{for $(p>0)$}\\
u\bullet_i(v\bullet_jl)=\left\{\begin{array}{ll}
v\bullet_j(u\bullet_{i+q-1}l)&\mbox{$j<i$}\\
v\bullet_{j+1}(u\bullet_il)&\mbox{if $0\leq i\leq j$}\\
\end{array}\right.  & \qquad\mbox{for $(p=0)$}\\
\end{array}
\end{equation*}
as well as $\mathbbm{1}\bullet_il=l$ for $i=0,1,...,n$. 

\smallskip
(c) A cyclic operator $t\in \underset{n\geq 1}{\prod} Vect_k( \mathscr L(n),\mathscr L'(n))$  satisfying
\begin{equation}
t(u\bullet_il)=u\bullet_{i+1}t(l) 
\end{equation} for $u\in \mathscr O(p)$, $l\in \mathscr L(n)$ and $0\leq i\leq n-p$ as well as $t^{n+1}=id$.

\end{defn}

We will now consider comodule measurings between cyclic unital comp modules.

\begin{defn}\label{D7.3}
Let $\{\Psi_n:C\longrightarrow Vect_k(\mathscr O(n),\mathscr O'(n))\}_{n\geq 0}$ be a measuring of non-$\Sigma$ operads with multiplication. Let $\mathscr L$ and $\mathscr L'$ be cyclic unital comp modules over $\mathscr O$ and $\mathscr O'$ respectively. Then, a    $(C,\{\Psi_n\}_{n\geq 0})$-comodule measuring 
from $\mathscr P$ to $\mathscr P'$ consists of a left $C$-comodule $D$ and a family of morphisms $\{\Omega_n:D\longrightarrow Vect_k(\mathscr L(n),\mathscr L'(n))\}_{n\geq 0}$ satisfying
\begin{equation}\label{7.7ub}
\Omega_{n-p+1}(y)(u\bullet_il)=\Psi_{p}(y_{(0)})(u)\bullet_i\Omega_{n}(y_{(1)})(l)
\end{equation}
for $y\in D$, $u\in\mathscr O(p)$, $l\in \mathscr L(n)$, $0\leq i\leq n+1-p$ and also
\begin{equation}\label{7.5rf}
\Omega_n(y)(t(l))=t'(\Omega_n(y)(l))
\end{equation} for $y\in D$, $l\in \mathscr L(n)$, where $t$ and $t'$ are respectively the cyclic operators on $\mathscr L$ and $\mathscr L'$. 

\end{defn}

\begin{thm}\label{P7.6vs}
Let $\{\Psi_n:C\longrightarrow Vect_k(\mathscr O(n),\mathscr O'(n))\}_{n\geq 0}$ be a measuring of non-$\Sigma$ operads with multiplication. Let $\mathscr L$ and $\mathscr L'$ be cyclic unital comp modules over $\mathscr O$ and $\mathscr O'$ respectively. Then, there exists a  $(C,\{\Psi_n\}_{n\geq 0})$-comodule  measuring $ 
\{\Theta_n:\mathcal Q_C(\mathscr L,\mathscr L')\longrightarrow Vect_k(\mathscr L(n),\mathscr L'(n))\}_{n\geq 0}$ satisfying the following universal property:  given any  $(C,\{\Psi_n\}_{n\geq 0})$-comodule measuring $
\{\Omega_n:D\longrightarrow Vect_k(\mathscr L(n),\mathscr L'(n))\}_{n\geq 0}$, there exists a morphism $\chi:D\longrightarrow \mathcal Q_C(\mathscr 
L,\mathscr L')$ of $C$-comodules making the following diagram commutative  for each $n\geq 0$
\begin{equation}
\xymatrix{
\mathcal Q_C(\mathscr L,\mathscr L')\ar[rr]^{\Theta_n} && Vect_k(\mathscr L(n),\mathscr L'(n))\\
&D\ar[ul]^\chi\ar[ur]_{\Omega_n}&\\
}
\end{equation}
\end{thm}

\begin{proof}
We set $V:\underset{n\geq 0}{\prod} Vect_k( \mathscr L(n),\mathscr L'(n))$. Using the adjunction in \eqref{radj5t}, we have the $C$-comodule $\mathfrak R_C(V)$ and the canonical morphism $\rho(V):\mathfrak R_C(V)\longrightarrow V$ of vector spaces. We now set $\mathcal Q_C(\mathscr L,\mathscr L')=\sum Q\subseteq \mathfrak R_C(V)$, where the sum is taken over all subcomodules $Q\subseteq \mathfrak R_C(V)$ such that the restriction of $\rho(V)$ to $Q$ gives a $(C,\{\Psi_n\}_{n\geq 0})$-comodule measuring from $\mathscr L$ to $\mathscr L'$. The universal property of $\mathcal Q_C(\mathscr P,\mathscr P')$ now follows as in the proof of 
Theorem \ref{T5.7}. 
\end{proof}

\begin{Thm}\label{T7.7kq}
Let $Comp_k$ be the category given by

\smallskip
(a) Objects: pairs $(\mathscr O,\mathscr L)$ where $\mathscr O$ is a non-$\Sigma$ operad with multiplication and $\mathscr L$ is a cyclic unital comp module over $\mathscr O$.

\smallskip
(b) Hom objects: for pairs $(\mathscr O,\mathscr L)$, $(\mathscr O',\mathscr L')\in Comp_k$, we set
\begin{equation}
Comp_k((\mathscr O,\mathscr L),(\mathscr O',\mathscr L')):=(\mathcal M_c(\mathscr O,\mathscr O'),\mathcal Q_{\mathcal M_c(\mathscr O,\mathscr O')}(\mathscr L,\mathscr L'))\in Comod_k
\end{equation} Then, $Comp_k$ is enriched over the symmetric monoidal category $Comod_k$. 
\end{Thm}

\begin{proof}
We consider  $(\mathscr O,\mathscr L)$, $(\mathscr O',\mathscr L')\in Comp_k$. For the measuring $\{\Phi_n:\mathcal M_c(\mathscr O,\mathscr O')\longrightarrow Vect_k(\mathscr O(n),\mathscr O'(n))\}_{n\geq 0}$ coalgebra, it folllows from Proposition \ref{P7.6vs} that we have 
the universal $(\mathcal M_c(\mathscr O,\mathscr O'),\{\Phi_n\}_{n\geq 0})$-comodule measuring  $ \{\Theta_n:\mathcal Q_{\mathcal M_c(\mathscr O,\mathscr O')}(\mathscr L,\mathscr L')\longrightarrow Vect_k(\mathscr L(n),\mathscr L'(n))\}_{n\geq 0}$. The scalar multiples of the identity map give a morphism $k\longrightarrow \mathcal M_c(\mathscr O,\mathscr O)$ of coalgebras, and using the universal property in Proposition \ref{P7.6vs} a morphism
$k\longrightarrow \mathcal Q_{\mathcal M_c(\mathscr O,\mathscr O')}(\mathscr L,\mathscr L')$. The composition of hom-objects now follows in a manner similar to the proof of Lemma 
\ref{Lem5.8} and Theorem \ref{T5.9hh}.
\end{proof}

Let $(\mathscr O,m,e)$ be a non-$\Sigma$ operad with multiplication and let $\mathscr L$ be a cyclic unital comp module over it.  We now recall from \cite[Proposition 3.5]{Ko4} that the cyclic homology of the pair $(\mathscr O,\mathscr L)$ is obtained from the cyclic module $C_\bullet(\mathscr O,\mathscr L):=\mathscr  L(\bullet)$ whose cyclic operators are $t:\mathscr L(\bullet)\longrightarrow \mathscr L(\bullet)$ and whose face maps and degeneracies are given as follows:
\begin{equation}\label{7.6rf}
d_i(l):=m\bullet_i l, \textrm{ }\mbox{$(0\leq i<n)$}\quad d_n(l):=m\bullet_0t(l)\qquad s_j(l):=e\bullet_{j
+1}l, \textrm{ }\mbox{$0\leq j\leq n$}
\end{equation} for $l\in \mathscr L(n)$. The cyclic homologies of this cyclic module will be denoted by $HC_\bullet(\mathscr O,\mathscr  L)$.  We now have the following result.

\begin{thm}\label{P7.5} Let $\{\Psi_n:C\longrightarrow Vect_k(\mathscr O(n),\mathscr O'(n))\}_{n\geq 0}$ be a measuring of non-$\Sigma$ operads with multiplication. Let $\mathscr L$ and $\mathscr L'$ be cyclic unital comp modules over $\mathscr O$ and $\mathscr O'$ respectively. 
Given a  $(C,\{\Psi_n\}_{n\geq 0})$-comodule measuring $ 
\{\Omega_n:D\longrightarrow Vect_k(\mathscr L(n),\mathscr L'(n))\}_{n\geq 0}$  from $\mathscr L$ to $\mathscr L'$, each $y\in D$ induces a morphism
$\underline\Omega^{cy}_\bullet(y):HC_\bullet(\mathscr O,\mathscr L)\longrightarrow HC_\bullet(\mathscr O',\mathscr L')$ on cyclic homologies.
\end{thm}

\begin{proof}
We know from \eqref{7.5rf} that the action of any $y\in D$ commutes with the cyclic operators. We know from \eqref{eq7.2ctn} that
$
x_2(m)=\Psi_2(x)(m)=\epsilon_C(x) m' $ and $ x_0(e)=\Psi_0(x)(e)=\epsilon_C(x)e'
$ for any $x\in C$. From the conditions in 
\eqref{7.7ub}  and the definitions in \eqref{7.6rf}, it is now easy to see that the action of of any $y\in D$ also commutes with the face maps and degeneracies. The result is now clear. 
\end{proof}

We conclude by studying measurings of Yetter Drinfeld algebras, which lead to a measuring of operads with multiplication, inducing morphisms in cyclic homology. We fix a Hopf algebroid
$\mathcal U=(U,A_L,s_L,t_L,\Delta_L,\epsilon_L,S)$.  In a manner similar to Definition \ref{D5.3}, a (left-left) Yetter Drinfeld module $P$ over $\mathcal U$ carries a left $U$-module and a left $U$-comodule structure:
\begin{equation}\label{713ij}
(u,p)\mapsto up \qquad p\mapsto p_{(-1)}\otimes p_{(0)} \qquad u\in U,\textrm{ }p\in P
\end{equation} satisfying compatibility conditions  for the   underlying $A^e$-module structure, as well as the $U$-action and $U$-coaction (see, for instance, \cite{Boh}, \cite[$\S$ 6.3.2]{Ko4}, 
\cite{Sch2}). The category ${^{\mathcal U}_{\mathcal U}}YD$
 of left-left Yetter Drinfeld modules is a braided monoidal category. A braided commutative Yetter-Drinfeld algebra  is a braided commutative monoid in ${^{\mathcal U}_{\mathcal U}}YD$ (see, for instance,
 \cite{BrzMi}, \cite[$\S$ 6.3.2]{Ko4}). For a  braided commutative Yetter-Drinfeld algebra $Z$ over the  Hopf algebroid $\mathcal U=(U,A_L,s_L,t_L,\Delta_L,\epsilon_L,S)$, the family 
 \begin{equation}\label{714rf}
 C^\bullet(\mathcal U,Z):=Hom_{A^{op}}(({_\blacktriangleright}U_\triangleleft)^{\otimes_{A^{op}}\bullet},Z)
 \end{equation} is a non-$\Sigma$ operad with multiplication (see \cite[$\S$ 6.3.3]{Ko4}). The structure maps are given by 
 \begin{equation}\label{715rf}
 \begin{array}{c}
 \circ_i: C^p(\mathcal U,Z)\otimes C^q(\mathcal U,Z)\longrightarrow C^{p+q-1}(\mathcal U,Z)\\
 (f\circ_ig)(u^1\otimes ...\otimes u^{p+q-1}):=f(u^1_{(1)}\otimes ...\otimes u^{p-i}_{(1)}\otimes [g(u^{p-i+1}_{(1)}\otimes ...\otimes u^{p+q-i}_{(1)})]_{(-1)}u_{(2)}^{p-i+1}...u_{(2)}^{p+q-i}\otimes u^{p+q-i+1}\otimes ...\otimes u^{p+q-1})\\
 \cdot_Z (u^1_{(2)}...u_{(2)}^{p-i}[g(u^{p-i+1}_{(1)}\otimes ...\otimes u^{p+q-i}_{(1)})]_{(0)})\\
 \end{array}
 \end{equation} where $\cdot_Z$ is the multiplication in $Z$. The multiplication $m\in C^2(\mathcal U,Z)$, the identity $\mathbbm{1}\in C^1(\mathcal U,Z)$ and the unit $e\in C^0(\mathcal U,Z)$ are given respectively by
 \begin{equation}\label{716rf}
 m:=\epsilon_L(\mu_U(\_\_\otimes \_\_))\triangleright 1_Z \qquad \mathbbm{1}:=\epsilon_L(\_\_)\triangleright 1_Z \qquad e:=1_Z
 \end{equation} where $\mu_U$ is the multiplication on $U$ and $1_Z$ is the identity in the Yetter Drinfeld algebra $Z$. If $Z$ and $Z'$ are braided commutative Yetter Drinfeld algebras over 
 $\mathcal U$, we say that a measuring from $Z$ to $Z'$ consists of a cocommutative coalgebra $C$ and a map $\psi: C\longrightarrow Vect_k(Z,Z')$, $\psi(x):=x(\_\_):Z\longrightarrow Z'$ satisfying the following conditions
 \begin{equation}\label{ms718}
x(z_1\cdot_Zz_2):=x_{(1)}(z_1)\cdot_{Z'}x_{(2)}(z_2)\qquad x(1_Z)=\epsilon_C(x)1_{Z'}\qquad x(uz)=ux(z)\qquad x(z)_{(-1)}\otimes x(z)_{(0)}=z_{(-1)}\otimes x(z_{(0)})
 \end{equation} for $x\in C$, $u\in U$ and $z_1$, $z_2$, $z\in Z$. 

\begin{lem}\label{L7.9uc}
Let $\mathcal U=(U,A_L,s_L,t_L,\Delta_L,\epsilon_L,S)$ be a Hopf algebroid. Let  $\psi:C\longrightarrow  Vect_k(Z,Z')$ be a measuring between braided commutative Yetter Drinfeld algebras over $\mathcal U$.  Then, $\psi$ induces a measuring $\{\Psi_n:C\longrightarrow Vect_k(C^n(\mathcal U,Z),C^n(\mathcal U,Z'))\}_{n\geq 0}$ of operads with multiplication.
\end{lem}
\begin{proof}
We define
\begin{equation}
\Psi_n:C\longrightarrow Vect_k(C^n(\mathcal U,Z),C^n(\mathcal U,Z'))\qquad x\mapsto (f\mapsto \psi(x)\circ f=xf)
\end{equation} for each $n\geq 0$. For $f\in C^p(\mathcal U,Z)$, $g\in C^q(\mathcal U,Z)$, we now have for each $x\in C$:
\begin{equation}
\begin{array}{l}
\Psi_{p+q-1}(x)(f\circ_ig)(u^1\otimes ...\otimes u^{p+q-1})\\
=x\left(
\begin{array}{l}f(u^1_{(1)}\otimes ...\otimes u^{p-i}_{(1)}\otimes [g(u^{p-i+1}_{(1)}\otimes ...\otimes u^{p+q-i}_{(1)})]_{(-1)}u_{(2)}^{p-i+1}...u_{(2)}^{p+q-i}\otimes u^{p+q-i+1}\otimes ...\otimes u^{p+q-1})\\
 \cdot_Z (u^1_{(2)}...u_{(2)}^{p-i}[g(u^{p-i+1}_{(1)}\otimes ...\otimes u^{p+q-i}_{(1)})]_{(0)})\\
\end{array}\right)\\
=\begin{array}{l} x_{(1)}(f(u^1_{(1)}\otimes ...\otimes u^{p-i}_{(1)}\otimes [g(u^{p-i+1}_{(1)}\otimes ...\otimes u^{p+q-i}_{(1)})]_{(-1)}u_{(2)}^{p-i+1}...u_{(2)}^{p+q-i}\otimes u^{p+q-i+1}\otimes ...\otimes u^{p+q-1}))\\
\cdot_{Z'} x_{(2)}(u^1_{(2)}...u_{(2)}^{p-i}[g(u^{p-i+1}_{(1)}\otimes ...\otimes u^{p+q-i}_{(1)})]_{(0)})\\
\end{array} \\
=\begin{array}{l} x_{(1)}(f(u^1_{(1)}\otimes ...\otimes u^{p-i}_{(1)}\otimes [g(u^{p-i+1}_{(1)}\otimes ...\otimes u^{p+q-i}_{(1)})]_{(-1)}u_{(2)}^{p-i+1}...u_{(2)}^{p+q-i}\otimes u^{p+q-i+1}\otimes ...\otimes u^{p+q-1}))\\
\cdot_{Z'} (u^1_{(2)}...u_{(2)}^{p-i}x_{(2)}([g(u^{p-i+1}_{(1)}\otimes ...\otimes u^{p+q-i}_{(1)})]_{(0)}))\\
\end{array} \\
=\begin{array}{l} x_{(1)}(f(u^1_{(1)}\otimes ...\otimes u^{p-i}_{(1)}\otimes [(x_{(2)}g)(u^{p-i+1}_{(1)}\otimes ...\otimes u^{p+q-i}_{(1)})]_{(-1)}u_{(2)}^{p-i+1}...u_{(2)}^{p+q-i}\otimes u^{p+q-i+1}\otimes ...\otimes u^{p+q-1}))\\
\cdot_{Z'} (u^1_{(2)}...u_{(2)}^{p-i} [(x_{(2)}g)(u^{p-i+1}_{(1)}\otimes ...\otimes u^{p+q-i}_{(1)})]_{(0)})\\
\end{array} \\
=\begin{array}{l} (x_{(1)}f)(u^1_{(1)}\otimes ...\otimes u^{p-i}_{(1)}\otimes [(x_{(2)}g)(u^{p-i+1}_{(1)}\otimes ...\otimes u^{p+q-i}_{(1)})]_{(-1)}u_{(2)}^{p-i+1}...u_{(2)}^{p+q-i}\otimes u^{p+q-i+1}\otimes ...\otimes u^{p+q-1})\\
\cdot_{Z'} (u^1_{(2)}...u_{(2)}^{p-i} [(x_{(2)}g)(u^{p-i+1}_{(1)}\otimes ...\otimes u^{p+q-i}_{(1)})]_{(0)})\\
\end{array} \\
=(\Psi_p(x_{(1)})(f))\circ_i\Psi_q(x_{(2)})(g))(u^1\otimes ...\otimes u^{p+q-1})\\
\end{array}
\end{equation} Let $m$ (resp. $m'$) and $e$ (resp. $e'$) be respectively the multiplication and unit on $C^\bullet(\mathcal U,Z)$ (resp. $C^\bullet(\mathcal U,Z')$) as in \eqref{716rf}. Then, we see that
\begin{equation}
\begin{array}{c}
\Psi_2(x)(m)(u^1\otimes u^2)=x(\epsilon_L(\mu_U(u^1\otimes u^2))\triangleright 1_Z)=\epsilon_L(\mu_U(u^1\otimes u^2))\triangleright \epsilon_C(x)1_{Z'}=\epsilon_C(x)m'(u^1\otimes u^2)\\ \Psi_0(x)(e)=x(1_Z)=\epsilon_C(x)1_{Z'}=\epsilon_
C(x)(e')  \\
\end{array}
\end{equation} This proves the result.
\end{proof}

On the other hand, if $L\in {^{\mathcal U}}aYD_{\mathcal U}$ is an anti-Yetter Drinfeld module and $Z$ is a braided commutative Yetter Drinfeld algebra, it follows from \cite[Lemma 6.1]{Ko4} that $L\otimes_{A^{op}}Z$ also carries a canonical anti-Yetter Drinfeld module structure. We now consider as in \cite[$\S$ 6.3.4]{Ko4}:
\begin{equation}\label{721op}
C_\bullet(\mathcal U,L\otimes_{A^{op}}Z):={_\blacktriangleright}({_\blacktriangleright}L\otimes_{A^{op}}Z_\triangleleft)\otimes_{A^{op}} {_\blacktriangleright}U_\triangleleft^{\otimes_{A^{op}}\bullet}
\end{equation} Additionally, if $ L\otimes_{A^{op}}Z$ is a stable anti-Yetter Drinfeld module, then $C_\bullet(\mathcal U,L\otimes_{A^{op}}Z)$ becomes a cyclic unital comp module over
$C^\bullet(\mathcal U,Z)$ with the following composition maps and cyclic operators:
\begin{equation}\label{72215}
\begin{array}{lll}
f\bullet_i(l\otimes z\otimes u^1\otimes ...\otimes u^k)&=l\otimes (u^1_{(2)}...u^{k-p-i+1}_{(2)}[f(u_{(1)}^{k-p-i+2}\otimes ...\otimes u_{(1)}^{k-i+1})]_{(0)})\cdot_Zz\otimes u^1_{(1)}\otimes ...\otimes u_{(1)}^{k-p-i+1}\otimes & \textrm{ }\mbox{$(i>0)$} \\
&\qquad [f(u_{(1)}^{k-p-i+2}\otimes ...\otimes u_{(1)}^{k-i+1})]_{(-1)})u^{k-p-i+2}_{(2)}...u_{(2)}^{k-i+1}\otimes u^{k-i+2}\otimes ... \otimes u^k & \\
f\bullet_0(l\otimes z\otimes u^1\otimes ...\otimes u^k)&=l_{(0)}\otimes ((u^1_{+(2)} ...  u^{k-p+1}_{+(2)}f(u_+^{k-p+2}\otimes ...\otimes u^k_+\otimes (u^k_-...u^1_-z_{(-1)}l_{(-1)}) ))\cdot_Zz_{(0)}) \otimes &\\
&\qquad u^1_{+(1)}\otimes ...\otimes u^{k-p+1}_{+(1)}&\\
t(l\otimes z\otimes u^1\otimes ...\otimes u^k)&=l_{(0)}u^1_{++}\otimes u^1_{+-}z_{(0)}\otimes u^2_+\otimes ...\otimes u^k_+\otimes u^k_-...u^1_-z_{(-1)}l_{(-1)}&\\
\end{array}
\end{equation} for $f\in C^p(\mathcal U,Z)$ and $l\otimes z\otimes u^1\otimes ...\otimes u^k\in C_k( \mathcal U,L\otimes_{A^{op}}Z)$. The cyclic homology of the corresponding cyclic module as in 
\eqref{7.6rf} will be denoted by $HC_\bullet(\mathcal U,Z,L)$. 

\smallskip Now let $h:L\longrightarrow L'$ be a morphism in ${^{\mathcal U}}aYD_{\mathcal U}$ and $\psi:C\longrightarrow  Vect_k(Z,Z')$ be a measuring of braided commutative Yetter Drinfeld algebras over $\mathcal U$. Accordingly, we have induced maps
\begin{equation}\label{722eop}
\begin{array}{c}
(h,\psi(x))_\bullet:C_\bullet(\mathcal U,L\otimes_{A^{op}}Z)={_\blacktriangleright}({_\blacktriangleright}L\otimes_{A^{op}}Z_\triangleleft)\otimes_{A^{op}} {_\blacktriangleright}U_\triangleleft^{\otimes_{A^{op}}\bullet}\longrightarrow {_\blacktriangleright}({_\blacktriangleright}L'\otimes_{A^{op}}Z'_\triangleleft)\otimes_{A^{op}} {_\blacktriangleright}U_\triangleleft^{\otimes_{A^{op}}\bullet}= C_\bullet(\mathcal U,L'\otimes_{A^{op}}Z')\\
l\otimes z\otimes u^1\otimes ...\otimes u^k\mapsto h(l)\otimes \psi(x)(z)\otimes u^1\otimes ...\otimes u^k\\
\end{array}
\end{equation} for each $x\in C$.  We end with the following result.

\begin{thm}\label{P7.9uc}
Let $\mathcal U=(U,A_L,s_L,t_L,\Delta_L,\epsilon_L,S)$ be a Hopf algebroid. Let  $\psi:C\longrightarrow  Vect_k(Z,Z')$ be a measuring between braided commutative Yetter Drinfeld algebras over $\mathcal U$.  Let $h:L\longrightarrow L'$ be a morphism in ${^{\mathcal U}}aYD_{\mathcal U}$. Suppose  that $L\otimes_{A^{op}}Z$ and $L'\otimes_{A^{op}}Z'$ are stable. Then, the induced morphisms:
\begin{equation}
\begin{array}{c}\Psi_\bullet:C\longrightarrow Vect_k(C^\bullet(\mathcal U,Z),C^\bullet(\mathcal U,Z')) \qquad  x\mapsto (f\mapsto \psi(x)\circ f)
\\
(h,\psi)_\bullet:C\longrightarrow Vect_k(C_\bullet(\mathcal U,L\otimes_{A^{op}}Z) , C_\bullet(\mathcal U,L'\otimes_{A^{op}}Z'))\qquad  x\mapsto (h,\psi(x))_\bullet 
\end{array}
\end{equation} give a comodule measuring of cyclic unital comp modules. In particular, each $x\in C$ induces a morphism $(h,\psi(x))_\bullet^{cy}:HC_\bullet(\mathcal U,Z,L)\longrightarrow HC_\bullet(\mathcal U,Z',L')$ on cyclic
homologies.
\end{thm}

\begin{proof} For $x\in C$, $f\in C^p(\mathcal U,Z)$, $l\otimes z\otimes u^1\otimes ...\otimes u^k\in C_k( U,L\otimes_{A^{op}}Z)$ and $i>0$, we see that
\begin{equation*}
\begin{array}{l}
(h,\psi(x))_{k-p+1}(f\bullet_i(l\otimes z\otimes u^1\otimes ...\otimes u^k))\\
=h(l)\otimes \psi(x)((u^1_{(2)}...u^{k-p-i+1}_{(2)}[f(u_{(1)}^{k-p-i+2}\otimes ...\otimes u_{(1)}^{k-i+1})]_{(0)})\cdot_Zz)\otimes u^1_{(1)}\otimes ...\otimes u_{(1)}^{k-p-i+1}\otimes \\
\qquad [f(u_{(1)}^{k-p-i+2}\otimes ...\otimes u_{(1)}^{k-i+1})]_{(-1)})u^{k-p-i+2}_{(2)}...u_{(2)}^{k-i+1}\otimes u^{k-i+2}\otimes ... \otimes u^k  \\
=h(l)\otimes x_{(1)}(u^1_{(2)}...u^{k-p-i+1}_{(2)}[f(u_{(1)}^{k-p-i+2}\otimes ...\otimes u_{(1)}^{k-i+1})]_{(0)}))\cdot_{Z'}x_{(2)}(z)\otimes u^1_{(1)}\otimes ...\otimes u_{(1)}^{k-p-i+1}\otimes \\
\qquad [f(u_{(1)}^{k-p-i+2}\otimes ...\otimes u_{(1)}^{k-i+1})]_{(-1)})u^{k-p-i+2}_{(2)}...u_{(2)}^{k-i+1}\otimes u^{k-i+2}\otimes ... \otimes u^k  \\
=h(l)\otimes  (u^1_{(2)}...u^{k-p-i+1}_{(2)}x_{(1)}([f(u_{(1)}^{k-p-i+2}\otimes ...\otimes u_{(1)}^{k-i+1})]_{(0)})))\cdot_{Z'}x_{(2)}(z)\otimes u^1_{(1)}\otimes ...\otimes u_{(1)}^{k-p-i+1}\otimes \\
\qquad [f(u_{(1)}^{k-p-i+2}\otimes ...\otimes u_{(1)}^{k-i+1})]_{(-1)})u^{k-p-i+2}_{(2)}...u_{(2)}^{k-i+1}\otimes u^{k-i+2}\otimes ... \otimes u^k  \\
=h(l)\otimes  (u^1_{(2)}...u^{k-p-i+1}_{(2)}[x_{(1)}f(u_{(1)}^{k-p-i+2}\otimes ...\otimes u_{(1)}^{k-i+1})]_{(0)}))\cdot_{Z'}x_{(2)}(z)\otimes u^1_{(1)}\otimes ...\otimes u_{(1)}^{k-p-i+1}\otimes \\
\qquad [x_{(1)}f(u_{(1)}^{k-p-i+2}\otimes ...\otimes u_{(1)}^{k-i+1})]_{(-1)})u^{k-p-i+2}_{(2)}...u_{(2)}^{k-i+1}\otimes u^{k-i+2}\otimes ... \otimes u^k  \\
=x_{(1)}f\bullet_i(h(l)\otimes x_{(2)}(z)\otimes u^1\otimes ...\otimes u^k)=x_{(1)}f\bullet_i (h,\psi(x_{(2)}))_{k}(l\otimes z\otimes u^1\otimes ...\otimes u^k)\\
\end{array}
\end{equation*} We notice that since $h:L\longrightarrow L'$ is a morphism in ${^{\mathcal U}}aYD_{\mathcal U}$, we must have $h(l)_{(-1)}\otimes h(l)_{(0)}=l_{(-1)}\otimes 
h(l_{(0)})$ for any $l\in L$. Using this fact, we now see that for $f\in C^p(\mathcal U,Z)$, $l\otimes z\otimes u^1\otimes ...\otimes u^k\in C_k(\mathcal U,L\otimes_{A^{op}}Z)$ and 
any $x\in C$, we have
\begin{equation*}
\begin{array}{l}
(h,\psi(x))_{k-p+1}(f\bullet_0(l\otimes z\otimes u^1\otimes ...\otimes u^k))\\
=h(l_{(0)})\otimes \psi(x)((u^1_{+(2)} ...  u^{k-p+1}_{+(2)}f(u_+^{k-p+2}\otimes ...\otimes u^k_+\otimes (u^k_-...u^1_-z_{(-1)}l_{(-1)}) ))\cdot_Zz_{(0)}) \otimes   u^1_{+(1)}\otimes ...\otimes u^{k-p+1}_{+(1)}\\
=h(l_{(0)})\otimes (x_{(1)}(u^1_{+(2)} ...  u^{k-p+1}_{+(2)}f(u_+^{k-p+2}\otimes ...\otimes u^k_+\otimes (u^k_-...u^1_-z_{(-1)}l_{(-1)}) ))\cdot_{Z'}x_{(2)}(z_{(0)})) \otimes   u^1_{+(1)}\otimes ...\otimes u^{k-p+1}_{+(1)}\\
=h(l)_{(0)}\otimes ((u^1_{+(2)} ...  u^{k-p+1}_{+(2)}x_{(1)}f(u_+^{k-p+2}\otimes ...\otimes u^k_+\otimes (u^k_-...u^1_-z_{(-1)}h(l)_{(-1)}) ))\cdot_{Z'}x_{(2)}(z_{(0)})) \otimes  u^1_{+(1)}\otimes ...\otimes u^{k-p+1}_{+(1)}\\
=h(l)_{(0)}\otimes ((u^1_{+(2)} ...  u^{k-p+1}_{+(2)}x_{(1)}f(u_+^{k-p+2}\otimes ...\otimes u^k_+\otimes (u^k_-...u^1_-x_{(2)}(z)_{(-1)}h(l)_{(-1)}) ))\cdot_{Z'}x_{(2)}(z)_{(0)}) \otimes   u^1_{+(1)}\otimes ...\otimes u^{k-p+1}_{+(1)}\\
=x_{(1)}f\bullet_0(h(l)\otimes x_{(2)}(z)\otimes u^1\otimes ...\otimes u^k)=x_{(1)}f\bullet_0 (h,\psi(x_{(2)}))_{k}(l\otimes z\otimes u^1\otimes ...\otimes u^k)\\
\end{array}
\end{equation*} Finally, we verify that the action of the maps $(h,\psi(x))_\bullet$ commutes with the cyclic operators, i.e., for any $l\otimes z\otimes u^1\otimes ...\otimes u^k\in C_k( \mathcal U,L\otimes_{A^{op}}Z)$, we have
\begin{equation*}
\begin{array}{ll}
(h,\psi(x))_k(t(l\otimes z\otimes u^1\otimes ...\otimes u^k))&=h(l_{(0)}u^1_{++})\otimes \psi(x)(u^1_{+-}z_{(0)})\otimes u^2_+\otimes ...\otimes u^k_+\otimes u^k_-...u^1_-z_{(-1)}l_{(-1)}\\
&=h(l_{(0)})u^1_{++}\otimes  (u^1_{+-}\psi(x)(z_{(0)}))\otimes u^2_+\otimes ...\otimes u^k_+\otimes u^k_-...u^1_-z_{(-1)}l_{(-1)}\\
&=h(l)_{(0)}u^1_{++}\otimes  (u^1_{+-}\psi(x)(z)_{(0)})\otimes u^2_+\otimes ...\otimes u^k_+\otimes u^k_-...u^1_-\psi(x)(z)_{(-1)}h(l)_{(-1)}\\
\end{array}
\end{equation*} It now follows that the collection $(h,\psi(x))_\bullet$ defines a measuring from $C_\bullet(\mathcal U,L\otimes_{A^{op}}Z)$ to $C_\bullet(\mathcal U,L'\otimes_{A^{op}}Z')$ in the sense 
of Definition \ref{D7.3}. The result is now clear by applying Proposition \ref{P7.5}.
\end{proof}

\end{document}